\documentclass{amsart}
\usepackage{amssymb,latexsym,amsmath,amscd,graphicx,graphics,epic,eepic,bm,color,array,mathrsfs,fullpage}
\usepackage{enumerate}
\usepackage[all,knot,poly]{xy}

\newcommand{\zed}{\mathbb{Z}}

\newcommand{\C}{\mathbb{C}}
\newcommand{\fil}{\mathcal{F}}

\newcommand{\ve}{\varepsilon}

\newcommand{\id}{\mathrm{id}}

\newcommand{\hmf}{\mathrm{hmf}}

\newcommand{\ch}{\mathsf{Ch}^{\mathsf{b}}}

\theoremstyle{plain}
\newtheorem{theorem}{Theorem}[section]
\newtheorem{lemma}[theorem]{Lemma}
\newtheorem{proposition}[theorem]{Proposition}
\newtheorem{corollary}[theorem]{Corollary}

\newtheorem{question}[theorem]{Question}

\theoremstyle{definition}

\newtheorem{acknowledgments}{Acknowledgments\ignorespaces}

\theoremstyle{remark}

\numberwithin{equation}{section}

\begin{document}

\title{Simplified Khovanov-Rozansky Chain Complexes of Open $2$-braids}

\author{Hao Wu}

\thanks{The author is partially supported by a Collaboration Grant for Mathematicians from the Simons Foundation.}

\address{Department of Mathematics, The George Washington University, Monroe Hall, Room 240, 2115 G Street, NW, Washington DC 20052, USA. Telephone: 1-202-994-0653, Fax: 1-202-994-6760}

\email{haowu@gwu.edu}



\begin{abstract}
Motivated by the works of Krasner \cite{Krasner-2-twists} and Lobb \cite{Lobb-2-twists}, we simplify the Khovanov-Rozansky chain complexes of open $2$-braids. As an application, we show that, for a knot containing a ``long" $2$-braid, the $\mathfrak{sl}(N)$ Rasmussen invariant of this knot depends linearly on the length of this $2$-braid. We refine this result for $(2,2k+1)$ cable knots and, as a simple corollary, compute the $\mathfrak{sl}(2)$ Rasmussen invariants of $(2,2k+1)$ cables of slice and amphicheiral knots.
\end{abstract}

\maketitle

\section{Introduction}\label{sec-intro}

For a non-negative integer $k$, denote by $T_k$ (resp. $T_{-k}$) the tangle formed by two parallel strands oriented in opposite directions with $k$ negative (resp. positive) full twists between them. (See Figure \ref{twisting-fig}, where $T_k$ contains $2k$ positive crossings and $T_{-k}$ contains $2k$ negative crossings.) Krasner \cite{Krasner-2-twists} constructed simplified Khovanov-Rozansky chain complexes of $T_{\pm k}$, which were used by Lobb \cite{Lobb-2-twists,Lobb-kanenobu-knots} to deduce new properties of the Khovanov-Rozansky homology and the $\mathfrak{sl}(N)$ Rasmussen invariant.

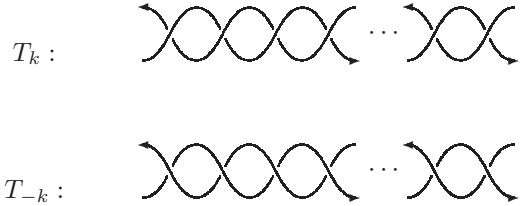
\begin{figure}[ht]
\[
\xymatrix{
T_k: & \setlength{\unitlength}{1pt}
\begin{picture}(140,20)(0,0)

\put(-2,20){\vector(-1,0){0}}

\qbezier(0,0)(5,0)(10,10)
\qbezier(10,10)(15,20)(20,20)

\qbezier(0,20)(5,20),(9,12)
\qbezier(11,8)(15,0),(20,0)

\qbezier(20,0)(25,0)(30,10)
\qbezier(30,10)(35,20)(40,20)

\qbezier(20,20)(25,20),(29,12)
\qbezier(31,8)(35,0),(40,0)

\qbezier(40,0)(45,0)(50,10)
\qbezier(50,10)(55,20)(60,20)

\qbezier(40,20)(45,20),(49,12)
\qbezier(51,8)(55,0),(60,0)

\qbezier(60,0)(65,0)(70,10)
\qbezier(70,10)(75,20)(80,20)

\qbezier(60,20)(65,20),(69,12)
\qbezier(71,8)(75,0),(80,0)

\put(82,0){\vector(1,0){0}}

\put(85,8){$\cdots$}

\put(98,20){\vector(-1,0){0}}

\qbezier(100,0)(105,0)(110,10)
\qbezier(110,10)(115,20)(120,20)

\qbezier(100,20)(105,20),(109,12)
\qbezier(111,8)(115,0),(120,0)

\qbezier(120,0)(125,0)(130,10)
\qbezier(130,10)(135,20)(140,20)

\qbezier(120,20)(125,20),(129,12)
\qbezier(131,8)(135,0),(140,0)

\put(142,0){\vector(1,0){0}}

\end{picture} \\
T_{-k}: & \setlength{\unitlength}{1pt}
\begin{picture}(140,20)(0,-20)

\put(-2,0){\vector(-1,0){0}}

\qbezier(0,0)(5,0)(10,-10)
\qbezier(10,-10)(15,-20)(20,-20)

\qbezier(0,-20)(5,-20),(9,-12)
\qbezier(11,-8)(15,0),(20,0)

\qbezier(20,0)(25,0)(30,-10)
\qbezier(30,-10)(35,-20)(40,-20)

\qbezier(20,-20)(25,-20),(29,-12)
\qbezier(31,-8)(35,0),(40,0)

\qbezier(40,0)(45,0)(50,-10)
\qbezier(50,-10)(55,-20)(60,-20)

\qbezier(40,-20)(45,-20),(49,-12)
\qbezier(51,-8)(55,0),(60,0)

\qbezier(60,0)(65,0)(70,-10)
\qbezier(70,-10)(75,-20)(80,-20)

\qbezier(60,-20)(65,-20),(69,-12)
\qbezier(71,-8)(75,0),(80,0)

\put(82,-20){\vector(1,0){0}}

\put(85,-12){$\cdots$}

\put(98,0){\vector(-1,0){0}}

\qbezier(100,0)(105,0)(110,-10)
\qbezier(110,-10)(115,-20)(120,-20)

\qbezier(100,-20)(105,-20),(109,-12)
\qbezier(111,-8)(115,0),(120,0)

\qbezier(120,0)(125,0)(130,-10)
\qbezier(130,-10)(135,-20)(140,-20)

\qbezier(120,-20)(125,-20),(129,-12)
\qbezier(131,-8)(135,0),(140,0)

\put(142,-20){\vector(1,0){0}}

\end{picture}
}
\]

\caption{Tangle diagrams $T_k$ and $T_{-k}$}\label{twisting-fig}

\end{figure}

Note that, reversing the orientation of one strand in $T_{\pm k}$, we get an open $2$-braid. Motivated by the results of Krasner and Lobb, we simplify the Khovanov-Rozansky chain complexes of $2$-braids and then apply the results to study the $\mathfrak{sl}(N)$ Rasmussen invariant.

Similar to \cite{Lobb-2-twists}, we consider three versions of the $\mathfrak{sl}(N)$ Khovanov-Rozansky chain complexes: the unreduced Khovanov-Rozansky chain complex $C_{x^{N+1}}$ with potential $x^{N+1}$ defined in \cite{KR1}, the equivariant Khovanov-Rozansky chain complex $C_{x^{N+1}-ax}$ with potential $x^{N+1}-ax$ defined in \cite{Krasner}, where $a$ is a variable of degree $2N$\footnote{Following the convention in \cite{KR1}, $x$ is a variable of degree $2$.}, and the deformed Khovanov-Rozansky chain complex $C_{x^{N+1}-x}$ with potential $x^{N+1}-x$ defined in \cite{Gornik}. The simplification of the chain complex works for all three versions of the Khovanov-Rozansky chain complexes. So, unless otherwise specified, ``the Khovanov-Rozansky chain complex $C$" stands for any of the above three versions. We denote by $H_{x^{N+1}}$, $H_{x^{N+1}-ax}$ and $H_{x^{N+1}-x}$ the homologies of $C_{x^{N+1}}$, $C_{x^{N+1}-ax}$ and $C_{x^{N+1}-x}$. Again, unless otherwise specified, ``the Khovanov-Rozansky homology $H$" stands for any of these three homologies. 

\begin{figure}[ht]
\[
\xymatrix{
\mathsf{b}^{2k}: & \setlength{\unitlength}{1pt}
\begin{picture}(160,20)(-10,-20)

\put(-10,0){\small{$x_3$}}

\put(-10,-20){\small{$x_4$}}

\qbezier(0,0)(5,0)(10,-10)
\qbezier(10,-10)(15,-20)(20,-20)

\qbezier(0,-20)(5,-20),(9,-12)
\qbezier(11,-8)(15,0),(20,0)

\qbezier(20,0)(25,0)(30,-10)
\qbezier(30,-10)(35,-20)(40,-20)

\qbezier(20,-20)(25,-20),(29,-12)
\qbezier(31,-8)(35,0),(40,0)

\qbezier(40,0)(45,0)(50,-10)
\qbezier(50,-10)(55,-20)(60,-20)

\qbezier(40,-20)(45,-20),(49,-12)
\qbezier(51,-8)(55,0),(60,0)

\qbezier(60,0)(65,0)(70,-10)
\qbezier(70,-10)(75,-20)(80,-20)

\qbezier(60,-20)(65,-20),(69,-12)
\qbezier(71,-8)(75,0),(80,0)

\put(82,-20){\vector(1,0){0}}

\put(82,0){\vector(1,0){0}}

\put(85,-12){$\cdots$}

\qbezier(100,0)(105,0)(110,-10)
\qbezier(110,-10)(115,-20)(120,-20)

\qbezier(100,-20)(105,-20),(109,-12)
\qbezier(111,-8)(115,0),(120,0)

\qbezier(120,0)(125,0)(130,-10)
\qbezier(130,-10)(135,-20)(140,-20)

\qbezier(120,-20)(125,-20),(129,-12)
\qbezier(131,-8)(135,0),(140,0)

\put(142,-20){\vector(1,0){0}}

\put(142,0){\vector(1,0){0}}

\put(145,0){\small{$x_1$}}

\put(145,-20){\small{$x_2$}}

\end{picture} \\
\mathsf{b}^{-2k}: & \setlength{\unitlength}{1pt}
\begin{picture}(160,20)(-10,0)

\put(-10,20){\small{$x_3$}}

\put(-10,0){\small{$x_4$}}

\qbezier(0,0)(5,0)(10,10)
\qbezier(10,10)(15,20)(20,20)

\qbezier(0,20)(5,20),(9,12)
\qbezier(11,8)(15,0),(20,0)

\qbezier(20,0)(25,0)(30,10)
\qbezier(30,10)(35,20)(40,20)

\qbezier(20,20)(25,20),(29,12)
\qbezier(31,8)(35,0),(40,0)

\qbezier(40,0)(45,0)(50,10)
\qbezier(50,10)(55,20)(60,20)

\qbezier(40,20)(45,20),(49,12)
\qbezier(51,8)(55,0),(60,0)

\qbezier(60,0)(65,0)(70,10)
\qbezier(70,10)(75,20)(80,20)

\qbezier(60,20)(65,20),(69,12)
\qbezier(71,8)(75,0),(80,0)

\put(82,0){\vector(1,0){0}}

\put(82,20){\vector(1,0){0}}

\put(85,8){$\cdots$}

\qbezier(100,0)(105,0)(110,10)
\qbezier(110,10)(115,20)(120,20)

\qbezier(100,20)(105,20),(109,12)
\qbezier(111,8)(115,0),(120,0)

\qbezier(120,0)(125,0)(130,10)
\qbezier(130,10)(135,20)(140,20)

\qbezier(120,20)(125,20),(129,12)
\qbezier(131,8)(135,0),(140,0)

\put(142,0){\vector(1,0){0}}

\put(142,20){\vector(1,0){0}}

\put(145,20){\small{$x_1$}}

\put(145,0){\small{$x_2$}}

\end{picture}
}
\]

\caption{Braid diagrams $\mathsf{b}^{2k}$ and $\mathsf{b}^{-2k}$}\label{braids-fig}

\end{figure}
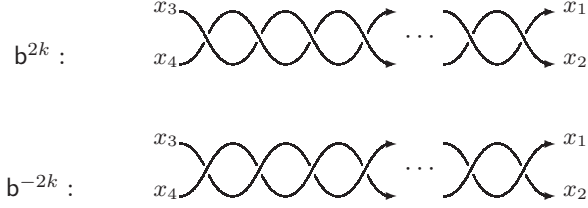

\subsection{The simplified chain complexes} Denote by $\mathcal{B}_2$ the $2$-strand braid group and by $\mathsf{b}$ the generator of $\mathcal{B}_2$ corresponding to a positive crossing. We mark the end points of $\mathsf{b}^{\pm 2k}$ as in Figure \ref{braids-fig}. 

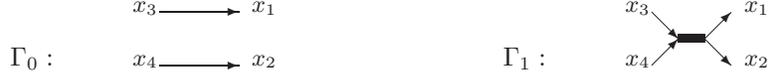
\begin{figure}[ht]
\[
\xymatrix{
\Gamma_0: & \setlength{\unitlength}{1pt}
\begin{picture}(50,20)(-10,0)

\put(-10,20){\small{$x_3$}}

\put(-10,0){\small{$x_4$}}

\put(0,0){\vector(1,0){30}}

\put(0,20){\vector(1,0){30}}

\put(35,20){\small{$x_1$}}

\put(35,0){\small{$x_2$}}

\end{picture} &&& \Gamma_1: & \setlength{\unitlength}{1pt}
\begin{picture}(50,20)(-10,0)

\put(-10,20){\small{$x_3$}}

\put(-10,0){\small{$x_4$}}

\put(0,0){\vector(1,1){10}}

\put(0,20){\vector(1,-1){10}}

\put(20,10){\vector(1,1){10}}

\put(20,10){\vector(1,-1){10}}

\put(35,20){\small{$x_1$}}

\put(35,0){\small{$x_2$}}

\linethickness{3pt}

\put(10,10){\line(1,0){10}}

\end{picture}
}
\]

\caption{MOY graphs $\Gamma_0$ and $\Gamma_1$}\label{resolutions-fig}

\end{figure}

Denote by $\Gamma_0$ and $\Gamma_1$ the marked MOY graphs in Figure \ref{resolutions-fig}. Recall that Khovanov and Rozansky defined in \cite[Section 6]{KR1} homogeneous homomorphisms of matrix factorizations $C(\Gamma_0) \xrightarrow{\chi^0} C(\Gamma_1)$ and $C(\Gamma_1) \xrightarrow{\chi^1} C(\Gamma_0)$ of $\zed_2$-degree $0$ and $q$-degree $1$.

For a graded matrix factorization $M$, denote by $M\{q^s\}$ the graded matrix factorization obtained from $M$ by shifting its $q$-grading up by $s$.

\begin{theorem}\label{thm-simplified-chain-complex}
Let $k$ be a positive integer. The Khovanov-Rozansky chain complex $C(\mathsf{b}^{2k})$ is homotopic to the chain complex $\mathscr{B}_k$ given by
\[
0 \rightarrow C(\Gamma_1)\{q^{s_{-2k}}\} \xrightarrow{x_1-x_3} C(\Gamma_1)\{q^{s_{-2k+1}}\}\xrightarrow{x_1-x_4} \cdots \xrightarrow{x_1-x_3} C(\Gamma_1)\{q^{s_{-1}}\} \xrightarrow{\chi^1} C(\Gamma_0)\{q^{2k(N-1)}\} \rightarrow 0,
\]
where 
\begin{itemize}
	\item $C(\Gamma_0)\{q^{2k(N-1)}\}$ has homological grading $0$;
	\item $C(\Gamma_1)\{q^{s_{-l}}\}$ has homological grading $-l$ and $s_{-l}=2k(N-1)+2l-1$;
	\item the differential map alternates between $x_1-x_3$ and $x_1-x_4$ from the homological grading $-2k$ to the homological grading $-2$. 
\end{itemize}
  
The Khovanov-Rozansky chain complex $C(\mathsf{b}^{-2k})$ is homotopic to the chain complex $\mathscr{B}_{-k}$ given by
\[
0 \rightarrow C(\Gamma_0)\{q^{-2k(N-1)}\} \xrightarrow{\chi^0} C(\Gamma_1)\{q^{s_{1}}\}\xrightarrow{x_1-x_3} C(\Gamma_1)\{q^{s_{2}}\}\xrightarrow{x_1-x_4} \cdots \xrightarrow{x_1-x_3} C(\Gamma_1)\{q^{s_{2k}}\} \rightarrow 0,
\]
where 
\begin{itemize}
	\item $C(\Gamma_0)\{q^{-2k(N-1)}\}$ has homological grading $0$;
	\item $C(\Gamma_1)\{q^{s_{l}}\}$ has homological grading $l$ and $s_{l}=-2k(N-1)-2l+1$;
	\item the differential map alternates between $x_1-x_3$ and $x_1-x_4$ from the homological grading $1$ to the homological grading $2k-1$. 
\end{itemize}
\end{theorem}

\subsection{A long exact sequence}\label{subsec-exact} Using Krasner's simplification of $C(T_{\pm k})$ in \cite{Krasner-2-twists}, Lobb \cite{Lobb-2-twists} introduced several new structural results about the Khovanov-Rozansky homology, including two directed systems of chain maps, some exact sequences and a generalization of the Khovanov-Rozansky homology that allows $T_{\pm \infty}$ to be part of a link diagram. Similar structures can easily be constructed using Theorem \ref{thm-simplified-chain-complex}. We leave the details to the reader. Here, we only construct in details a long exact sequence of the Khovanov-Rozansky homology, which we will use to study the $\mathfrak{sl}(N)$ Rasmussen invariant. 

Fix a positive integer $k$. Consider the chain complex $\mathscr{B}_k$ in Theorem \ref{thm-simplified-chain-complex}. We define a chain map 
\[
F_k: \mathscr{B}_{k-1}\{q^{2(N-1)}\} \rightarrow \mathscr{B}_{k}
\] 
using the diagram
{\tiny
\[
\xymatrix{
&& 0 \ar[r] & C(\Gamma_1) \ar[r]^{x_1-x_3} \ar[d]^{\id} & \cdots \ar[r]^{x_1-x_3} &  C(\Gamma_1) \ar[r]^{\chi^1} \ar[d]^{\id}& C(\Gamma_0) \ar[r] \ar[d]^{\id}& 0 \\
0 \ar[r] & C(\Gamma_1) \ar[r]^{x_1-x_3} & C(\Gamma_1) \ar[r]^{x_1-x_4} & C(\Gamma_1) \ar[r]^{x_1-x_3}& \cdots \ar[r]^{x_1-x_3} &  C(\Gamma_1) \ar[r]^{\chi^1} & C(\Gamma_0) \ar[r] & 0,
}
\]
}

\noindent where the upper row is $\mathscr{B}_{k-1}\{q^{2(N-1)}\}$, the lower row is $\mathscr{B}_{k}$ and $F_k$ is given by the vertical arrows, which are all identity maps. (For simplicity, we omitted the $q$-grading shifts in the above diagram.)

Denote by $\mathscr{C}$ the mapping cone\footnote{In this paper, the homological grading of the mapping cone of $A^\ast \rightarrow B^\ast$ is defined so that the term $A^j\oplus B^{j-1}$ has homological grading $j$.} of $C(\Gamma_1) \xrightarrow{x_1-x_3} C(\Gamma_1)\{q^{-2}\}$. Then, by the definition of $F_k$, there is a short exact sequence

\begin{equation}\label{short-exact}
0 \rightarrow \mathscr{B}_{k-1}\{q^{2(N-1)}\} \xrightarrow{F_k} \mathscr{B}_{k} \rightarrow \mathscr{C}\|-2k\|\{q^{2k(N+1)-1}\} \rightarrow 0,
\end{equation}
where ``$\|-2k\|$" means that the homological grading of $\mathscr{C}$ is shifted down by $2k$.

\begin{figure}[ht]
\[
\xymatrix{
D_0: &  \setlength{\unitlength}{1pt}
\begin{picture}(80,20)(-10,0)

\put(-10,20){\small{$x_3$}}

\put(-10,0){\small{$x_4$}}

\put(0,0){\vector(1,0){60}}

\put(0,20){\vector(1,0){60}}

\put(65,20){\small{$x_1$}}

\put(65,0){\small{$x_2$}}

\end{picture} \\
D_k: & \setlength{\unitlength}{1pt}
\begin{picture}(80,20)(-10,-20)

\put(-10,0){\small{$x_3$}}

\put(-10,-20){\small{$x_4$}}

\qbezier(0,0)(5,0)(10,-10)
\qbezier(10,-10)(15,-20)(20,-20)

\qbezier(0,-20)(5,-20),(9,-12)
\qbezier(11,-8)(15,0),(20,0)

\put(22,-20){\vector(1,0){0}}

\put(22,0){\vector(1,0){0}}

\put(25,-12){$\cdots$}

\qbezier(40,0)(45,0)(50,-10)
\qbezier(50,-10)(55,-20)(60,-20)

\qbezier(40,-20)(45,-20),(49,-12)
\qbezier(51,-8)(55,0),(60,0)

\put(62,-20){\vector(1,0){0}}

\put(62,0){\vector(1,0){0}}

\put(65,0){\small{$x_1$}}

\put(65,-20){\small{$x_2$}}

\end{picture} \\ 
\Gamma: & \input{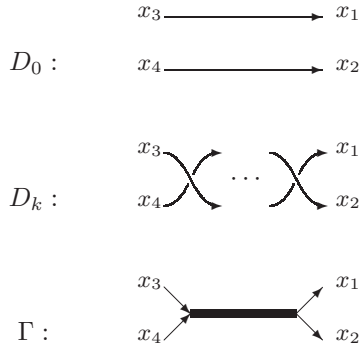}
}
\]

\caption{Local diagrams of links $D_0$, $D_k$ and the knotted MOY graph $\Gamma$}\label{D-k-Gamma-fig}

\end{figure}

Let $D_0$ be a link diagram containing a pair of parallel arcs in the same direction. Denote by $D_k$ the link diagram obtained from $D_0$ by replacing this pair of parallel arcs with the open $2$-braid $\mathsf{b}^{2k}$. Also, denote by $\Gamma$ the knotted MOY graph obtained from $D_0$ by replacing this pair of parallel arcs with a wide edge. (See Figure \ref{D-k-Gamma-fig}. Also note the markings we put on $D_0$, $D_k$ and $\Gamma$.) 

Denote by $\mathscr{C}(\Gamma)$ the mapping cone of $C(\Gamma) \xrightarrow{x_1-x_3} C(\Gamma)\{q^{-2}\}$
and by $\mathscr{H}(\Gamma)$ the homology of $\mathscr{C}(\Gamma)$.

\begin{theorem}\label{thm-exact-sequence}
There are long exact sequences
\begin{equation}\label{long-exact-sequence}
\cdots \rightarrow H^{j}(D_{k-1})\{q^{2(N-1)}\} \rightarrow H^{j}(D_{k}) \rightarrow \mathscr{H}^{j+2k}(\Gamma)\{q^{2k(N+1)-1}\} \rightarrow H^{j+1}(D_{k-1})\{q^{2(N-1)}\} \rightarrow \cdots,
\end{equation}
and
\begin{equation}\label{minor-long-exact-sequence}
\cdots \rightarrow H^{j-1}(\Gamma)\{q^{-2}\} \rightarrow \mathscr{H}^{j}(\Gamma) \rightarrow H^{j}(\Gamma) \xrightarrow{x_1-x_3} H^{j}(\Gamma)\{q^{-2}\} \rightarrow \cdots,
\end{equation}
where $H^{j}(D_{k})$ is the component of $H(D_{k})$ of homological grading $j$ and so on.
\end{theorem}
\begin{proof}
Long exact sequence \eqref{long-exact-sequence} is the long exact sequence in homology associated to the short exact sequence \eqref{short-exact} after making the identification $\mathscr{B}_k \simeq C(\mathsf{b}^{2k})$ of Theorem 1.1. Long exact sequence \eqref{minor-long-exact-sequence} is the long exact sequence in homology associated to the short exact sequence of the mapping cone $\mathscr{C}(\Gamma)$.
\end{proof}

\subsection{The Rasmussen invariants} Rasmussen \cite{Ras1} defined the concordance invariant $s$ using the deformed Khovanov homology constructed by Lee \cite{Lee2}. Based on the work of Gornik \cite{Gornik}, the author generalized $s$ in \cite{Wu7} to an $\mathfrak{sl}(N)$ Rasmussen invariant $s_N$ with $s_2=s$. Recently, Lobb proved in \cite{Lobb-gornik} that $s_N$ is a concordance invariant for every $N \geq 2$.

Assume that $D_0$ in Subsection \ref{subsec-exact} is a knot diagram. Then $D_k$ is also a knot diagram. Using Theorem \ref{thm-exact-sequence}, we prove that, as $|k|$ increases, $s_N(D_k)$ quickly becomes linearly dependent on $k$. 

\begin{theorem}\label{thm-linearity-general}
Assume that $D_0$ is a knot diagram with $c_+$ positive crossings and $c_-$ negative crossings. Then, for all $N\geq 2$,
\begin{equation}\label{eq-linearity-general}
s_N(D_k) = s_N(D_{k-1}) + 2(N-1) \text{ if } k \leq \frac{-c_+ -2}{2} \text{ or } k\geq \frac{c_- +2}{2}.
\end{equation}
\end{theorem}

The behavior under cabling of the concordance invariant $\tau$ defined by  Ozsv\'ath, Szab\'o \cite{OS-4-ball} and Rasmussen \cite{Ras-thesis} is well-understood through the works of Hedden \cite{Hedden-cable-1,Hedden-cable-2}, Van Cott \cite{VanCott} and Hom \cite{Hom-thesis}. We know much less about $s_N$. Though, we do know from \cite[Theorem 9]{VanCott} and \cite[Theorem 1.7]{Lobb-gornik} that, for any knot $K$ and relatively prime integers $m$ and $n$, the $\mathfrak{sl}(N)$ Rasmussen invariant of the $(m,lm+n)$ cable of $K$ is linear with respect to $l$ for large $l$.

Next, we refine Theorem \ref{thm-linearity-general} for $(2,2k+1)$ cable knots and conclude that, for these cables, the linearity of $s_N$ starts rather quickly. This is especially true for $s_2$, which allows a simple computation of $s_2$ for $(2,2k+1)$ cables of slice or amphicheiral knots.

\begin{theorem}\label{thm-linearity-2cable}
Suppose that $K$ is a knot with a diagram that has $c_+$ positive crossings and $c_-$ negative crossings. Denote by $K_{2,2k+1}$ the $(2,2k+1)$ cable of $K$. Then, for all $N\geq 2$,
\begin{equation}\label{eq-linearity-2cable-N}
s_N(K_{2,2k+1}) = s_N(K_{2,2k-1}) + 2(N-1) \text{ if } k \leq -c_+-1 \text{ or } k\geq c_-+1.
\end{equation}
Moreover,
\begin{equation}\label{eq-linearity-2cable-2}
s_2(K_{2,2k+1}) = s_2(K_{2,2k-1}) + 2 \text{ if } k \neq 0.
\end{equation}
\end{theorem}

By equation \eqref{eq-linearity-2cable-2}, one can compute $s_2(K_{2,2k+1})$ for all $k$ if the values of $s_2(K_{2,\pm 1})$ are known. Unfortunately, this is not a simple problem in general. But, when $K$ is slice or amphicheiral, $s_2(K_{2,\pm 1})$ is easy to determine.

\begin{corollary}\label{cor-ras-amphicheiral}
If a knot $K$ is slice or (positively or negatively) amphicheiral, then $s_2(K_{2,2k+1}) =2k$ and $s_2(K_{2,-(2k+1)}) = - 2k$ for all $k\geq0$.
\end{corollary}

\begin{figure}[ht]
\[
\xymatrix{
K \cup K' : &  \setlength{\unitlength}{1pt}
\begin{picture}(80,20)(-10,0)

\put(0,0){\vector(1,0){60}}

\put(0,20){\vector(1,0){60}}

\end{picture} \\
K_{2,1}: & \setlength{\unitlength}{1pt}
\begin{picture}(80,20)(-10,-20)

\put(0,0){\line(1,0){20}}

\put(0,-20){\line(1,0){20}}

\qbezier(20,0)(25,0)(30,-10)
\qbezier(30,-10)(35,-20)(40,-20)

\qbezier(20,-20)(25,-20),(29,-12)
\qbezier(31,-8)(35,0),(40,0)

\put(40,0){\vector(1,0){20}}

\put(40,-20){\vector(1,0){20}}

\end{picture} \\ 
K_{2,-1}: & \setlength{\unitlength}{1pt}
\begin{picture}(80,20)(-10,0)

\put(0,0){\line(1,0){20}}

\put(0,20){\line(1,0){20}}

\qbezier(20,0)(25,0)(30,10)
\qbezier(30,10)(35,20)(40,20)

\qbezier(20,20)(25,20),(29,12)
\qbezier(31,8)(35,0),(40,0)

\put(40,0){\vector(1,0){20}}

\put(40,20){\vector(1,0){20}}

\end{picture}
}
\]

\caption{Local diagrams of $K \cup K'$, $K_{2,1}$ and $K_{2,-1}$ }\label{slice-cable-fig}

\end{figure}
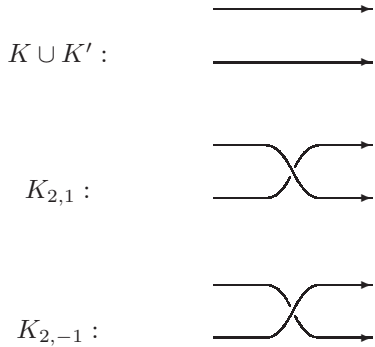

\begin{proof}
Assume $K$ is slice. Let $\Delta$ be a slice disk of $K$ properly smoothly embedded in $B^4$ and $\vec{v}$ a non-vanishing vector field on $\Delta$ that is everywhere transverse to $\Delta$ and tangent to $S^3$ along $K$. Of course, $\vec{v}|_K$ gives the Seifert framing of $K$. Push $\Delta$ slightly in the direction of $\vec{v}$. This gives a properly smoothly embedded disk $\Delta'$ in $B^4$ parallel to $\Delta$, whose boundary is a push-off $K'$ of $K$ along its Seifert framing. As in Figure \ref{slice-cable-fig}, we insert a positive (resp. negative) crossing in a pair of parallel arcs in $K$ and $K'$ by gluing a twisted band between $\Delta$ and $\Delta'$. This changes the link $K \cup K'$ into $K_{2,1}$ (resp. $K_{2,-1}$) and combines $\Delta$ and $\Delta'$ into a slice disk of $K_{2,1}$ (resp. $K_{2,-1}$). So $K_{2,1}$ and $K_{2,-1}$ are both slice. By \cite[Theorem 1]{Ras1}, it follows that $s_2(K_{2,1}) =s_2(K_{2,-1})=0$, which, by equation \eqref{eq-linearity-2cable-2}, implies the corollary for slice knots. 

Assume $K$ is (positively or negatively) amphicheiral. Denote by $\overline{K}$ the mirror image of $K$ and $-K$ the inverse of $K$. Then $K = \pm \overline{K}$. Thus, $K_{2,1} = (\pm \overline{K})_{2,1} = \pm \overline{K_{2,-1}}$. By \cite[Theorem 2]{Ras1}, this implies $s_2(K_{2,1}) = -s_2(K_{2,-1})$. But $s_2(K_{2,1})$ and $s_2(K_{2,-1})$ are both even integers and,  by \cite[Corollary 4.3]{Ras1}, $s_2(K_{2,1}) -s_2(K_{2,-1})=0 \text{ or } 2$. For all these to be simultaneously true, we must have $s_2(K_{2,1}) =s_2(K_{2,-1})=0$, which, by equation \eqref{eq-linearity-2cable-2}, implies the corollary for amphicheiral knots.
\end{proof}

Since the $\mathfrak{sl}(N)$ Rasmussen invariants share many properties, one may conjecture that these are just simple linear re-normalizations of the original $\mathfrak{sl}(2)$ Rasmussen invariant. (See, for example, \cite[Conjecture 1.6]{Lobb-gornik}.) In the proof of Theorem \ref{thm-linearity-2cable} in Section \ref{sec-ras}, we will see that equation \eqref{eq-linearity-2cable-2} is established using the fact that the $2$-colored $\mathfrak{sl}(2)$ Khovanov-Rozansky homology is trivial. The $2$-colored $\mathfrak{sl}(N)$ Khovanov-Rozansky homology is non-trivial for $N>2$. So it seems unlikely that one can improve equation \eqref{eq-linearity-2cable-N} to be as strong as equation \eqref{eq-linearity-2cable-2}. 

\begin{question}\label{que-independence-s_N}
Does this mean that one can disprove \cite[Conjecture 1.6]{Lobb-gornik} by computing the $\mathfrak{sl}(N)$ Rasmussen invariants of $(2,2k+1)$ cable knots?
\end{question}

\subsection{Organization of this paper} In Section \ref{sec-simplify}, we simplify the $\mathfrak{sl}(N)$ Khovanov-Rozansky chain complexes of open $2$-braids and prove Theorem \ref{thm-simplified-chain-complex}. Then we use Theorem \ref{thm-exact-sequence} to prove Theorems \ref{thm-linearity-general} and \ref{thm-linearity-2cable} in Section \ref{sec-ras}. 

We assume the reader is familiar with the Khovanov-Rozansky homology defined in \cite{KR1}. Some knowledge of \cite{Gornik,Krasner,Lobb-gornik,Lobb-2-twists,Wu7,Wu-color,Wu-color-equi} would also help.

\begin{acknowledgments}
The author would like to thank the referee for providing many helpful suggestions.
\end{acknowledgments}

\section{Simplifying the Chain Complexes}\label{sec-simplify}

In this section, we simplify the $\mathfrak{sl}(N)$ Khovanov-Rozansky chain complexes of $\mathsf{b}^{\pm 2k}$ and prove Theorem \ref{thm-simplified-chain-complex}. We will do so by repeatedly applying the second MOY decomposition. To do this, we need explicit descriptions of the inclusion and projection maps. 

\begin{figure}[ht]
\[
\xymatrix{
\Gamma: & \setlength{\unitlength}{1pt}
\begin{picture}(80,20)(-10,0)

\put(-10,20){\small{$x_3$}}

\put(-10,0){\small{$x_4$}}

\put(0,0){\vector(1,1){10}}

\put(0,20){\vector(1,-1){10}}

\qbezier(20,10)(30,30)(40,10)

\qbezier(20,10)(30,-10)(40,10)

\put(40,10){\vector(1,-2){0}}

\put(40,10){\vector(1,2){0}}

\put(30,19){\line(0,1){2}}

\put(30,-1){\line(0,1){2}}

\put(25,22){\small{$x_5$}}

\put(25,-6){\small{$x_6$}}

\put(50,10){\vector(1,1){10}}

\put(50,10){\vector(1,-1){10}}

\put(65,20){\small{$x_1$}}

\put(65,0){\small{$x_2$}}

\linethickness{3pt}

\put(10,10){\line(1,0){10}}

\put(40,10){\line(1,0){10}}

\end{picture} &&& \Gamma': &   \input{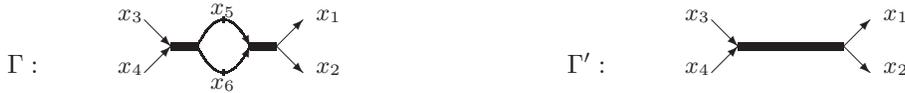}
}
\]

\caption{MOY graphs $\Gamma$ and $\Gamma'$}\label{MOY-decomp-II-fig}

\end{figure}

\subsection{The second MOY decomposition} We start by reviewing the second MOY decomposition. Let $\Gamma$ and $\Gamma'$ be the MOY graphs in Figure \ref{MOY-decomp-II-fig}. Khovanov and Rozansky proved in \cite{KR1} the following homotopy equivalence, which we call the second MOY decomposition. 

\begin{proposition}\cite[Proposition 30]{KR1}\label{MOY-decomp-II}
\[
C(\Gamma) \simeq C(\Gamma')\{q\} \oplus C(\Gamma')\{q^{-1}\}.
\]
\end{proposition}

Next, we explicitly construct the inclusion and projection maps in this decomposition. Let us first recall some basic homomorphisms of matrix factorizations.

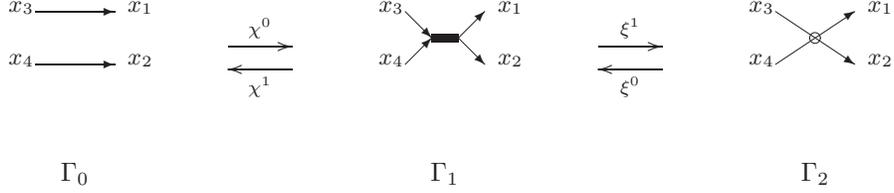
\begin{figure}[ht]
\[
\xymatrix{
\setlength{\unitlength}{1pt}
\begin{picture}(50,20)(-10,0)

\put(-10,20){\small{$x_3$}}

\put(-10,0){\small{$x_4$}}

\put(0,0){\vector(1,0){30}}

\put(0,20){\vector(1,0){30}}

\put(35,20){\small{$x_1$}}

\put(35,0){\small{$x_2$}}

\end{picture} &\ar@<1ex>[r]^{\chi^0} & \ar@<1ex>[l]^{\chi^1} & \setlength{\unitlength}{1pt}
\begin{picture}(50,20)(-10,0)

\put(-10,20){\small{$x_3$}}

\put(-10,0){\small{$x_4$}}

\put(0,0){\vector(1,1){10}}

\put(0,20){\vector(1,-1){10}}

\put(20,10){\vector(1,1){10}}

\put(20,10){\vector(1,-1){10}}

\put(35,20){\small{$x_1$}}

\put(35,0){\small{$x_2$}}

\linethickness{3pt}

\put(10,10){\line(1,0){10}}

\end{picture} & \ar@<1ex>[r]^{\xi^1} & \ar@<1ex>[l]^{\xi^0} & \setlength{\unitlength}{1pt}
\begin{picture}(50,20)(-10,0)

\put(-10,20){\small{$x_3$}}

\put(-10,0){\small{$x_4$}}

\put(0,0){\vector(3,2){30}}

\put(0,20){\vector(3,-2){30}}

\put(15,10){\circle{4}}

\put(35,20){\small{$x_1$}}

\put(35,0){\small{$x_2$}}

\end{picture} \\
\Gamma_0 &&& \Gamma_1 &&& \Gamma_2
}
\]
\caption{Homomorphisms $\chi^0$, $\chi^1$, $\xi^0$ and $\xi^1$}\label{chi-xi-fig}

\end{figure}

Let $\Gamma_0$, $\Gamma_1$ and $\Gamma_2$ be the MOY graphs in Figure \ref{chi-xi-fig}, where the circled crossing in $\Gamma_2$ is a virtual crossing as introduced in \cite{Gornik}. That is, $C(\Gamma_2)$ is the tensor product over the base ring of the matrix factorizations of the two arcs in $\Gamma_2$. Khovanov and Rozansky \cite{KR1} defined the homomorphisms $\chi^0$ and $\chi^0$ between $C(\Gamma_0)$ and $C(\Gamma_1)$. Gornik \cite{Gornik} defined the homomorphisms $\xi^1$ and $\xi^0$ between $C(\Gamma_1)$ and $C(\Gamma_2)$. The following are basic properties of these homomorphisms.
\begin{itemize}
	\item $C(\Gamma_0) \xrightarrow{\chi^0} C(\Gamma_1)$, $C(\Gamma_1) \xrightarrow{\chi^1} C(\Gamma_0)$, $C(\Gamma_1) \xrightarrow{\xi^1} C(\Gamma_2)$ and $C(\Gamma_2) \xrightarrow{\xi^0} C(\Gamma_1)$ are $\C[x_1,x_2,x_3,x_4]$-linear homogeneous homomorphisms of $\zed_2$-degree $0$ and $q$-degree $1$.
	\item 
	\begin{eqnarray}
	\label{eq-chi-comp} \chi^1 \circ \chi^0 = (x_1-x_4) \id_{C(\Gamma_0)}, & \chi^0 \circ \chi^1 = (x_1-x_4)\id_{C(\Gamma_1)}, \\
	\label{eq-xi-comp} \xi^1 \circ \xi^0 = (x_1-x_3) \id_{C(\Gamma_2)}, & \xi^0 \circ \xi^1 = (x_1-x_3) \id_{C(\Gamma_1)}.
	\end{eqnarray}
	\item Up to homotopy and scaling, $\chi^0$, $\chi^1$, $\xi^0$ and $\xi^1$ are the unique homotopically non-trivial $\C[x_1,x_2,x_3,x_4]$-linear homomorphisms of $q$-degree $1$ between these matrix factorizations.
\end{itemize}

\begin{lemma}\label{lemma-chi-xi-compose=0}
	\begin{eqnarray*}
	\xi^1 \circ \chi^0 & \simeq & 0, \\
	\chi^1 \circ \xi^0 & \simeq & 0.
	\end{eqnarray*}
\end{lemma}
\begin{proof}
By \cite[Lemma 7.33]{Wu-color} and \cite[Lemma 4.15]{Wu-color-equi}, all homotopically non-trivial homomorphisms from $C(\Gamma_0)$ to $C(\Gamma_2)$ or from $C(\Gamma_2)$ to $C(\Gamma_0)$ have $\zed_2$-degree $1$. But both $\xi^1 \circ \chi^0$ and $\chi^1 \circ \xi^0$ have $\zed_2$-degree $0$. So they are both homotopic to $0$.
\end{proof}

\begin{figure}[ht]
\[
\xymatrix{
\Gamma'': & \setlength{\unitlength}{1pt}
\begin{picture}(80,20)(-10,0)

\put(-10,20){\small{$x_3$}}

\put(-10,0){\small{$x_4$}}

\put(0,0){\vector(1,1){10}}

\put(0,20){\vector(1,-1){10}}

\qbezier(20,10)(35,25)(50,10)

\qbezier(20,10)(35,-5)(50,10)

\put(50,10){\circle{4}}

\put(35,16.5){\line(0,1){2}}

\put(35,1.5){\line(0,1){2}}

\put(30,22){\small{$x_5$}}

\put(30,-6){\small{$x_6$}}

\put(50,10){\vector(1,1){10}}

\put(50,10){\vector(1,-1){10}}

\put(65,20){\small{$x_1$}}

\put(65,0){\small{$x_2$}}

\linethickness{3pt}

\put(10,10){\line(1,0){10}}

\end{picture}
}
\]

\caption{$\Gamma''$}\label{trick-fig}

\end{figure}

Let $\Gamma$ and $\Gamma'$ be as in Figure \ref{MOY-decomp-II-fig} and $\Gamma''$ be as in Figure \ref{trick-fig}. Then we get the homomorphisms $C(\Gamma) \xrightarrow{\chi_r^1} C(\Gamma')$, $C(\Gamma') \xrightarrow{\chi_r^0} C(\Gamma)$, $C(\Gamma) \xrightarrow{\xi_r^1} C(\Gamma'')$ and $C(\Gamma'') \xrightarrow{\xi_r^0} C(\Gamma)$, where the lower index ``$r$" means that these homomorphisms are associated to the right wide edge in $\Gamma$. There is an obvious homotopy equivalence $C(\Gamma'') \simeq C(\Gamma')$, which we use to identify $C(\Gamma'')$ and $C(\Gamma')$. This allows us to view $\xi^0$ and $\xi^1$ as homomorphisms between $C(\Gamma)$ and $C(\Gamma')$.

By \cite[Lemma 7.9]{Wu-color} and \cite[Lemma 4.8]{Wu-color-equi}, up to homotopy and scaling, there are a unique homogeneous $\C[x_1,x_2,x_3,x_4]$-linear homomorphism $J:C(\Gamma'') (\simeq C(\Gamma')) \rightarrow C(\Gamma)$ of $q$-degree $-1$ and a unique homogeneous $\C[x_1,x_2,x_3,x_4]$-linear homomorphism $P:C(\Gamma) \rightarrow C(\Gamma'') (\simeq C(\Gamma'))$ of $q$-degree $-1$. By \cite[Lemma 7.11]{Wu-color} and \cite[Lemma 4.10]{Wu-color-equi}, after possibly scaling $J$ and $P$, we have
\begin{eqnarray}
\label{comp-J-P-x-5} P \circ \mathtt{m}(x_5) \circ J & \simeq & \id_{C(\Gamma')}, \\
\label{comp-J-P-x-6} P \circ \mathtt{m}(x_6) \circ J & \simeq & -\id_{C(\Gamma')}, \\
\label{comp-J-P-1} P \circ J & \simeq & 0,
\end{eqnarray}
where $\mathtt{m}(\bullet): C(\Gamma) \rightarrow C(\Gamma)$ is the homomorphism given by the multiplication of $\bullet$.

\begin{lemma}\label{lemma-J-P-chi-xi-compose}
Up to scaling by non-zero scalars,
\[
P\circ \chi_r^0 \simeq \chi_r^1 \circ J \simeq P\circ \xi_r^0 \simeq \xi_r^1 \circ J \simeq \id_{C(\Gamma')}.
\]
\end{lemma}

\begin{proof}
It is easy to check that the space of homotopy classes of homomorphisms from $C(\Gamma')$ to $C(\Gamma')$ of $q$-degree $0$ is $1$-dimensional and spanned by $\id_{C(\Gamma')}$. (See for example \cite[Proof of Lemma 4.4]{Wu7}.) So, to prove the lemma, we only need to show that these compositions are not homotopic to $0$. Using equations \eqref{eq-chi-comp}, \eqref{comp-J-P-x-6} and \eqref{comp-J-P-1}, we have
\[
P\circ \chi_r^0 \circ \chi_r^1 \circ J = P\circ \mathtt{m}(x_1-x_6) \circ J = \mathtt{m}(x_1) \circ P \circ J - P\circ \mathtt{m}(x_6) \circ J \simeq \id_{C(\Gamma')}.
\]
So $P\circ \chi_r^0$ and $\chi_r^1 \circ J$ are not homotopic to $0$. Using equations \eqref{eq-xi-comp}, \eqref{comp-J-P-x-5}, and \eqref{comp-J-P-1}, one can similarly show that $P\circ \xi_r^0$ and $\xi_r^1 \circ J$ are not homotopic to 0.
\end{proof}

\begin{lemma}\label{lemma-homotopy-inverses}
The homomorphisms
\begin{equation}\label{homotopy-inverses-1}
\xymatrix{
{\left.%
\begin{array}{c}
  C(\Gamma')\{q\} \\
  \oplus \\
  C(\Gamma')\{q^{-1}\} 
\end{array}%
\right.} & \ar@<1ex>[r]^{(\chi_r^0, J)}
&\ar@<1ex>[l]^{\left(%
\begin{array}{c}
  P \\
  \xi_r^1 \\
\end{array}%
\right)}
&
C(\Gamma) 
}
\end{equation}
are homotopy equivalences and are homotopy inverses of each other.

Similarly, the homomorphisms
\begin{equation}\label{homotopy-inverses-2}
\xymatrix{
{\left.%
\begin{array}{c}
  C(\Gamma')\{q\} \\
  \oplus \\
  C(\Gamma')\{q^{-1}\} 
\end{array}%
\right.} & \ar@<1ex>[r]^{(\xi_r^0, J)}
&\ar@<1ex>[l]^{\left(%
\begin{array}{c}
  P \\
  \chi_r^1 \\
\end{array}%
\right)}
&
C(\Gamma) 
}
\end{equation}
are homotopy equivalences and are homotopy inverses of each other.
\end{lemma}

\begin{proof}
We only prove the lemma for the homomorphisms in \eqref{homotopy-inverses-1}. The proof for the homomorphisms in \eqref{homotopy-inverses-2} is similar and left to the reader.

By Lemmas \ref{lemma-chi-xi-compose=0}, \ref{lemma-J-P-chi-xi-compose} and the homotopy \eqref{comp-J-P-1}, we have that
\[
\left(%
\begin{array}{c}
  P \\
  \xi_r^1 \\
\end{array}%
\right) \circ
(\chi_r^0, J) \simeq 
\left(%
\begin{array}{cc}
  \id_{C(\Gamma')\{q\}} & 0 \\
  0 & \id_{C(\Gamma')\{q^{-1}\}} \\
\end{array}%
\right).
\] 
Recall that $C(\Gamma) \simeq C(\Gamma')\{q\} \oplus C(\Gamma')\{q^{-1}\}$ and the homotopy category $\hmf$ of graded matrix factorizations with finite-dimensional cohomology is Krull-Schmidt (see \cite[Proposition 25]{KR1}). So the above implies that the homomorphisms in \eqref{homotopy-inverses-1} are homotopy equivalences and are homotopy inverses of each other.
\end{proof}

\subsection{Simplifying $C(\mathsf{b}^{2k})$} The proofs of $C(\mathsf{b}^{2k}) \simeq \mathscr{B}_k$ and $C(\mathsf{b}^{-2k}) \simeq \mathscr{B}_{-k}$ are very similar. We only give the details of the former and leave the latter to the reader. 

First, we recall the Gaussian elimination formulated by Bar-Natan in \cite{Bar-fast}.

\begin{lemma}\cite[Lemma 4.2]{Bar-fast}\label{gaussian-elimination}
Let $\mathcal{C}$ be an additive category, and
\[
\mathtt{I}=``\cdots\rightarrow C\xrightarrow{\left(%
\begin{array}{c}
  \alpha\\
  \beta \\
\end{array}%
\right)}
\left.%
\begin{array}{c}
  A\\
  \oplus \\
  D
\end{array}%
\right.
\xrightarrow{
\left(%
\begin{array}{cc}
  \phi & \delta\\
  \gamma & \varepsilon \\
\end{array}%
\right)}
\left.%
\begin{array}{c}
  B\\
  \oplus \\
  E
\end{array}%
\right.
\xrightarrow{
\left(%
\begin{array}{cc}
  \mu & \nu\\
\end{array}%
\right)} F \rightarrow \cdots"
\]
an object of $\ch(\mathcal{C})$, that is, a bounded chain complex over $\mathcal{C}$. Assume that $A\xrightarrow{\phi} B$ is an isomorphism in $\mathcal{C}$ with inverse $\phi^{-1}$. Then $\mathtt{I}$ is homotopic to 
\[
\mathtt{II}=
``\cdots\rightarrow C \xrightarrow{\beta} D
\xrightarrow{\varepsilon-\gamma\phi^{-1}\delta} E\xrightarrow{\nu} F \rightarrow \cdots".
\]
In particular, if $\delta$ or $\gamma$ is $0$, then $\mathtt{I}$ is homotopic to 
\[
\mathtt{II}=
``\cdots\rightarrow C \xrightarrow{\beta} D
\xrightarrow{\varepsilon} E\xrightarrow{\nu} F \rightarrow \cdots".
\]
\end{lemma}

\begin{lemma}\label{lemma-b+2} 
$C(\mathsf{b}^{2}) \simeq \mathscr{B}_1$
\end{lemma}

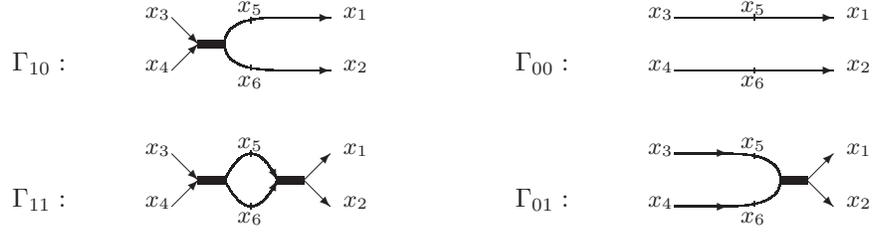
\begin{figure}[ht]
\[
\xymatrix{
\Gamma_{10}: & \setlength{\unitlength}{1pt}
\begin{picture}(80,20)(-10,0)

\put(-10,20){\small{$x_3$}}

\put(-10,0){\small{$x_4$}}

\put(0,0){\vector(1,1){10}}

\put(0,20){\vector(1,-1){10}}

\qbezier(20,10)(20,20)(40,20)

\qbezier(20,10)(20,0)(40,0)

\put(30,18){\line(0,1){2}}

\put(30,0){\line(0,1){2}}

\put(25,22){\small{$x_5$}}

\put(25,-6){\small{$x_6$}}

\put(40,20){\vector(1,0){20}}

\put(40,0){\vector(1,0){20}}

\put(65,20){\small{$x_1$}}

\put(65,0){\small{$x_2$}}

\linethickness{3pt}

\put(10,10){\line(1,0){10}}

\end{picture} && \Gamma_{00}: & \setlength{\unitlength}{1pt}
\begin{picture}(80,20)(-10,0)

\put(-10,20){\small{$x_3$}}

\put(-10,0){\small{$x_4$}}

\put(0,0){\vector(1,0){60}}

\put(0,20){\vector(1,0){60}}

\put(65,20){\small{$x_1$}}

\put(65,0){\small{$x_2$}}

\put(30,19){\line(0,1){2}}

\put(30,-1){\line(0,1){2}}

\put(25,22){\small{$x_5$}}

\put(25,-6){\small{$x_6$}}

\end{picture} \\
\Gamma_{11}: & \setlength{\unitlength}{1pt}
\begin{picture}(80,20)(-10,0)

\put(-10,20){\small{$x_3$}}

\put(-10,0){\small{$x_4$}}

\put(0,0){\vector(1,1){10}}

\put(0,20){\vector(1,-1){10}}

\qbezier(20,10)(30,30)(40,10)

\qbezier(20,10)(30,-10)(40,10)

\put(40,10){\vector(1,-2){0}}

\put(40,10){\vector(1,2){0}}

\put(30,19){\line(0,1){2}}

\put(30,-1){\line(0,1){2}}

\put(25,22){\small{$x_5$}}

\put(25,-6){\small{$x_6$}}

\put(50,10){\vector(1,1){10}}

\put(50,10){\vector(1,-1){10}}

\put(65,20){\small{$x_1$}}

\put(65,0){\small{$x_2$}}

\linethickness{3pt}

\put(10,10){\line(1,0){10}}

\put(40,10){\line(1,0){10}}

\end{picture} && \Gamma_{01}: & \setlength{\unitlength}{1pt}
\begin{picture}(80,20)(-10,0)

\put(-10,20){\small{$x_3$}}

\put(-10,0){\small{$x_4$}}

\put(0,0){\vector(1,0){20}}

\put(0,20){\vector(1,0){20}}

\qbezier(20,20)(40,20)(40,10)

\qbezier(20,0)(40,0)(40,10)

\put(30,18){\line(0,1){2}}

\put(30,0){\line(0,1){2}}

\put(25,22){\small{$x_5$}}

\put(25,-6){\small{$x_6$}}

\put(50,10){\vector(1,1){10}}

\put(50,10){\vector(1,-1){10}}

\put(65,20){\small{$x_1$}}

\put(65,0){\small{$x_2$}}

\linethickness{3pt}

\put(40,10){\line(1,0){10}}

\end{picture} \\
}
\]

\caption{MOY resolutions of $\mathsf{b}^{2}$}\label{chain-b+2-fig}

\end{figure}

\begin{proof}
Let $\Gamma_{00}$, $\Gamma_{01}$, $\Gamma_{10}$ and $\Gamma_{11}$ be the MOY graphs in Figure \ref{chain-b+2-fig}. Then, by definition $C(\mathsf{b}^{2})$ is the chain complex
\[
0 \rightarrow C(\Gamma_{11})\{q^{2N}\} \xrightarrow{\left(%
\begin{array}{c}
  \chi_r^1 \\
  \chi_l^1 \\
\end{array}%
\right)} {\left.%
\begin{array}{c}
  C(\Gamma_{10})\{q^{2N-1}\} \\
  \oplus \\
  C(\Gamma_{01})\{q^{2N-1}\} 
\end{array}%
\right.}
\xrightarrow{(-\chi_l^1,\chi_r^1)}
C(\Gamma_{00})\{q^{2N-2}\} \rightarrow 0,
\]
where the lower index ``$r$" (resp. ``$l$") means the homomorphism is associated to the wide edge on the right side (resp. left side). In this chain complex, we replace $C(\Gamma_{11})\{q^{2N}\}$ with $\left.%
\begin{array}{c}
  C(\Gamma_{10})\{q^{2N-1}\} \\
  \oplus \\
  C(\Gamma'')\{q^{2N+1}\} 
\end{array}%
\right.$ using the homotopy equivalence
\[
\left.%
\begin{array}{c}
  C(\Gamma_{10})\{q^{2N-1}\} \\
  \oplus \\
  C(\Gamma'')\{q^{2N+1}\} 
\end{array}%
\right. 
\xrightarrow{(J,\xi_r^0)}
C(\Gamma_{11})\{q^{2N}\}
\]
given in Lemma \ref{lemma-homotopy-inverses}, where $\Gamma''$ is the MOY graph in Figure \ref{trick-fig}. 

Thus, $C(\mathsf{b}^{2})$ is isomorphic to the chain complex
\[
0 \rightarrow \left.%
\begin{array}{c}
  C(\Gamma_{10})\{q^{2N-1}\} \\
  \oplus \\
  C(\Gamma'')\{q^{2N+1}\}
\end{array}%
\right. \xrightarrow{\left(%
\begin{array}{cc}
  \chi_r^1 \circ J & \chi_r^1\circ\xi_r^0 \\
  \chi_l^1 \circ J & \chi_l^1\circ\xi_r^0 \\
\end{array}%
\right)} {\left.%
\begin{array}{c}
  C(\Gamma_{10})\{q^{2N-1}\} \\
  \oplus \\
  C(\Gamma_{01})\{q^{2N-1}\} 
\end{array}%
\right.}
\xrightarrow{(-\chi_l^1,\chi_r^1)}
C(\Gamma_{00})\{q^{2N-2}\} \rightarrow 0.
\]
From Lemmas \ref{lemma-chi-xi-compose=0} and \ref{lemma-J-P-chi-xi-compose}, we have $\chi_r^1\circ\xi_r^0 \simeq 0$ and $\chi_r^1 \circ J \simeq \id$. So, by Lemma \ref{gaussian-elimination}, $C(\mathsf{b}^{2})$ is homotopic to the chain complex
\begin{equation}\label{chain-almost-B-1}
0 \rightarrow  C(\Gamma'')\{q^{2N+1}\} \xrightarrow{\chi_l^1\circ\xi_r^0} C(\Gamma_{01})\{q^{2N-1}\} \xrightarrow{\chi_r^1} C(\Gamma_{00})\{q^{2N-2}\} \rightarrow 0.
\end{equation}

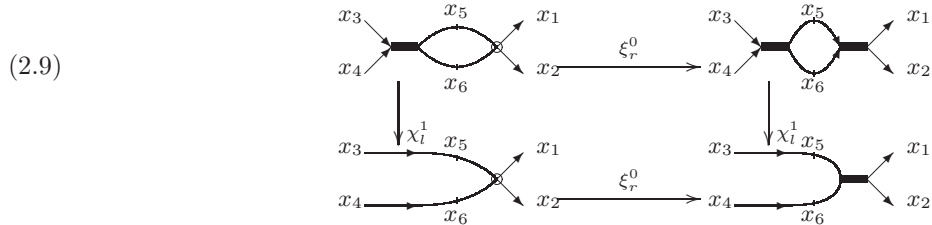
\begin{equation}\label{lemma-b+2-diagram-1}
\xymatrix{
\setlength{\unitlength}{1pt}
\begin{picture}(80,20)(-10,0)

\put(-10,20){\small{$x_3$}}

\put(-10,0){\small{$x_4$}}

\put(0,0){\vector(1,1){10}}

\put(0,20){\vector(1,-1){10}}

\qbezier(20,10)(35,25)(50,10)

\qbezier(20,10)(35,-5)(50,10)

\put(50,10){\circle{4}}

\put(35,16.5){\line(0,1){2}}

\put(35,1.5){\line(0,1){2}}

\put(30,22){\small{$x_5$}}

\put(30,-6){\small{$x_6$}}

\put(50,10){\vector(1,1){10}}

\put(50,10){\vector(1,-1){10}}

\put(65,20){\small{$x_1$}}

\put(65,0){\small{$x_2$}}

\linethickness{3pt}

\put(10,10){\line(1,0){10}}

\end{picture} \ar[rr]^{\xi_r^0} \ar@<-4ex>[d]^{\chi_l^1} && \setlength{\unitlength}{1pt}
\begin{picture}(80,20)(-10,0)

\put(-10,20){\small{$x_3$}}

\put(-10,0){\small{$x_4$}}

\put(0,0){\vector(1,1){10}}

\put(0,20){\vector(1,-1){10}}

\qbezier(20,10)(30,30)(40,10)

\qbezier(20,10)(30,-10)(40,10)

\put(40,10){\vector(1,-2){0}}

\put(40,10){\vector(1,2){0}}

\put(30,19){\line(0,1){2}}

\put(30,-1){\line(0,1){2}}

\put(25,22){\small{$x_5$}}

\put(25,-6){\small{$x_6$}}

\put(50,10){\vector(1,1){10}}

\put(50,10){\vector(1,-1){10}}

\put(65,20){\small{$x_1$}}

\put(65,0){\small{$x_2$}}

\linethickness{3pt}

\put(10,10){\line(1,0){10}}

\put(40,10){\line(1,0){10}}

\end{picture} \ar@<-4ex>[d]^{\chi_l^1} \\
\setlength{\unitlength}{1pt}
\begin{picture}(80,20)(-10,0)

\put(-10,20){\small{$x_3$}}

\put(-10,0){\small{$x_4$}}

\put(0,0){\vector(1,0){20}}

\put(0,20){\vector(1,0){20}}

\qbezier(20,20)(40,20)(50,10)

\qbezier(20,0)(40,0)(50,10)

\put(50,10){\circle{4}}

\put(35,17){\line(0,1){2}}

\put(35,1){\line(0,1){2}}

\put(30,22){\small{$x_5$}}

\put(30,-6){\small{$x_6$}}

\put(50,10){\vector(1,1){10}}

\put(50,10){\vector(1,-1){10}}

\put(65,20){\small{$x_1$}}

\put(65,0){\small{$x_2$}}

\end{picture} \ar[rr]^{\xi_r^0} && \setlength{\unitlength}{1pt}
\begin{picture}(80,20)(-10,0)

\put(-10,20){\small{$x_3$}}

\put(-10,0){\small{$x_4$}}

\put(0,0){\vector(1,0){20}}

\put(0,20){\vector(1,0){20}}

\qbezier(20,20)(40,20)(40,10)

\qbezier(20,0)(40,0)(40,10)

\put(30,18){\line(0,1){2}}

\put(30,0){\line(0,1){2}}

\put(25,22){\small{$x_5$}}

\put(25,-6){\small{$x_6$}}

\put(50,10){\vector(1,1){10}}

\put(50,10){\vector(1,-1){10}}

\put(65,20){\small{$x_1$}}

\put(65,0){\small{$x_2$}}

\linethickness{3pt}

\put(40,10){\line(1,0){10}}

\end{picture}
}
\end{equation}
Diagram \eqref{lemma-b+2-diagram-1} commutes since $\chi_l^1$ and $\xi_r^0$ act on disjoint parts of the MOY graphs. That is, 
\begin{equation}\label{eq-xi-chi-lr-commute}
\chi_l^1\circ\xi_r^0 = \xi_r^0 \circ \chi_l^1.
\end{equation}

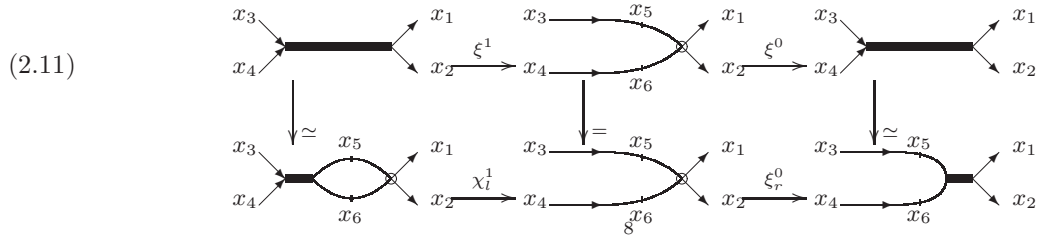
\begin{equation}\label{lemma-b+2-diagram-2}
\xymatrix{
\setlength{\unitlength}{1pt}
\begin{picture}(80,20)(-10,0)

\put(-10,20){\small{$x_3$}}

\put(-10,0){\small{$x_4$}}

\put(0,0){\vector(1,1){10}}

\put(0,20){\vector(1,-1){10}}

\put(50,10){\vector(1,1){10}}

\put(50,10){\vector(1,-1){10}}

\put(65,20){\small{$x_1$}}

\put(65,0){\small{$x_2$}}

\linethickness{3pt}

\put(10,10){\line(1,0){40}}

\end{picture} \ar[r]^{\xi^1} \ar@<-4ex>[d]^{\simeq} & \setlength{\unitlength}{1pt}
\begin{picture}(80,20)(-10,0)

\put(-10,20){\small{$x_3$}}

\put(-10,0){\small{$x_4$}}

\put(0,0){\vector(1,0){20}}

\put(0,20){\vector(1,0){20}}

\qbezier(20,20)(40,20)(50,10)

\qbezier(20,0)(40,0)(50,10)

\put(50,10){\circle{4}}

\put(35,17){\line(0,1){2}}

\put(35,1){\line(0,1){2}}

\put(30,22){\small{$x_5$}}

\put(30,-6){\small{$x_6$}}

\put(50,10){\vector(1,1){10}}

\put(50,10){\vector(1,-1){10}}

\put(65,20){\small{$x_1$}}

\put(65,0){\small{$x_2$}}

\end{picture} \ar[r]^{\xi^0} \ar@<-4ex>[d]^{=} & \setlength{\unitlength}{1pt}
\begin{picture}(80,20)(-10,0)

\put(-10,20){\small{$x_3$}}

\put(-10,0){\small{$x_4$}}

\put(0,0){\vector(1,1){10}}

\put(0,20){\vector(1,-1){10}}

\put(50,10){\vector(1,1){10}}

\put(50,10){\vector(1,-1){10}}

\put(65,20){\small{$x_1$}}

\put(65,0){\small{$x_2$}}

\linethickness{3pt}

\put(10,10){\line(1,0){40}}

\end{picture} \ar@<-4ex>[d]^{\simeq} \\
\setlength{\unitlength}{1pt}
\begin{picture}(80,20)(-10,0)

\put(-10,20){\small{$x_3$}}

\put(-10,0){\small{$x_4$}}

\put(0,0){\vector(1,1){10}}

\put(0,20){\vector(1,-1){10}}

\qbezier(20,10)(35,25)(50,10)

\qbezier(20,10)(35,-5)(50,10)

\put(50,10){\circle{4}}

\put(35,16.5){\line(0,1){2}}

\put(35,1.5){\line(0,1){2}}

\put(30,22){\small{$x_5$}}

\put(30,-6){\small{$x_6$}}

\put(50,10){\vector(1,1){10}}

\put(50,10){\vector(1,-1){10}}

\put(65,20){\small{$x_1$}}

\put(65,0){\small{$x_2$}}

\linethickness{3pt}

\put(10,10){\line(1,0){10}}

\end{picture}  \ar[r]^{\chi_l^1} & \setlength{\unitlength}{1pt}
\begin{picture}(80,20)(-10,0)

\put(-10,20){\small{$x_3$}}

\put(-10,0){\small{$x_4$}}

\put(0,0){\vector(1,0){20}}

\put(0,20){\vector(1,0){20}}

\qbezier(20,20)(40,20)(50,10)

\qbezier(20,0)(40,0)(50,10)

\put(50,10){\circle{4}}

\put(35,17){\line(0,1){2}}

\put(35,1){\line(0,1){2}}

\put(30,22){\small{$x_5$}}

\put(30,-6){\small{$x_6$}}

\put(50,10){\vector(1,1){10}}

\put(50,10){\vector(1,-1){10}}

\put(65,20){\small{$x_1$}}

\put(65,0){\small{$x_2$}}

\end{picture} \ar[r]^{\xi_r^0} & \setlength{\unitlength}{1pt}
\begin{picture}(80,20)(-10,0)

\put(-10,20){\small{$x_3$}}

\put(-10,0){\small{$x_4$}}

\put(0,0){\vector(1,0){20}}

\put(0,20){\vector(1,0){20}}

\qbezier(20,20)(40,20)(40,10)

\qbezier(20,0)(40,0)(40,10)

\put(30,18){\line(0,1){2}}

\put(30,0){\line(0,1){2}}

\put(25,22){\small{$x_5$}}

\put(25,-6){\small{$x_6$}}

\put(50,10){\vector(1,1){10}}

\put(50,10){\vector(1,-1){10}}

\put(65,20){\small{$x_1$}}

\put(65,0){\small{$x_2$}}

\linethickness{3pt}

\put(40,10){\line(1,0){10}}

\end{picture}
}
\end{equation}
Next we identify both $C(\Gamma'')$ and $C(\Gamma_{01})$ with $C(\Gamma')$ in Figure \ref{MOY-decomp-II-fig} using the apparent homotopy equivalences and consider diagram \eqref{lemma-b+2-diagram-2}. Recall that, up to homotopy and scaling, $\xi^1$ and $\xi^0$ are the unique homotopically non-trivial homomorphisms of $q$-degree $1$ and that $\chi_l^1$ and $\xi_r^0$ are also homotopically non-trivial homomorphisms of $q$-degree $1$. This implies that diagram \eqref{lemma-b+2-diagram-2} commutes up to homotopy and scaling. Thus, there is a non-zero scalar $c\in \C$ such that  
\begin{equation}\label{eq-switch-to-xi}
\xi_r^0 \circ \chi_l^1 \simeq c \cdot \xi^0 \circ \xi^1 = c \cdot (x_1-x_3) \cdot \id_{C(\Gamma')}.
\end{equation}

Putting equations \eqref{eq-xi-chi-lr-commute} and \eqref{eq-switch-to-xi} together, we know that chain complex \eqref{chain-almost-B-1} is isomorphic to
\[
0 \rightarrow  C(\Gamma')\{q^{2N+1}\} \xrightarrow{c(x_1-x_3)} C(\Gamma')\{q^{2N-1}\} \xrightarrow{\chi_r^1} C(\Gamma_{00})\{q^{2N-2}\} \rightarrow 0,
\]
which is clearly isomorphic to $\mathscr{B}_1$.
\end{proof}

\begin{proof}[Proof of Theorem \ref{thm-simplified-chain-complex}] 
As mentioned above, we only prove that 
\begin{equation}\label{homotopy-simplified-chain-complex+2k}
C(\mathsf{b}^{2k}) \simeq \mathscr{B}_k.
\end{equation}
The proof of $C(\mathsf{b}^{-2k}) \simeq \mathscr{B}_{-k}$ is very similar and left to the reader.

We prove \eqref{homotopy-simplified-chain-complex+2k} by an induction on $k$. The $k=1$ case is already established in Lemma \ref{lemma-b+2}. Assume \eqref{homotopy-simplified-chain-complex+2k} is true for a given $k\geq 1$. Consider $C(\mathsf{b}^{2k+2})$.

\begin{figure}[ht]
\[
\xymatrix{
\mathsf{b}^{2k+2}: & \setlength{\unitlength}{1pt}
\begin{picture}(160,20)(-10,-20)

\put(-10,0){\small{$x_3$}}

\put(-10,-20){\small{$x_4$}}

\qbezier(0,0)(5,0)(10,-10)
\qbezier(10,-10)(15,-20)(20,-20)

\qbezier(0,-20)(5,-20),(9,-12)
\qbezier(11,-8)(15,0),(20,0)

\qbezier(20,0)(25,0)(30,-10)
\qbezier(30,-10)(35,-20)(40,-20)

\qbezier(20,-20)(25,-20),(29,-12)
\qbezier(31,-8)(35,0),(40,0)

\qbezier(40,0)(45,0)(50,-10)
\qbezier(50,-10)(55,-20)(60,-20)

\qbezier(40,-20)(45,-20),(49,-12)
\qbezier(51,-8)(55,0),(60,0)

\put(62,-20){\vector(1,0){0}}

\put(62,0){\vector(1,0){0}}

\put(65,-12){$\cdots$}

\qbezier(80,0)(85,0)(90,-10)
\qbezier(90,-10)(95,-20)(100,-20)

\qbezier(80,-20)(85,-20),(89,-12)
\qbezier(91,-8)(95,0),(100,0)

\put(100,-1){\line(0,1){2}}

\put(100,-21){\line(0,1){2}}

\put(95,2){\small{$x_5$}}

\put(95,-26){\small{$x_6$}}

\qbezier(100,0)(105,0)(110,-10)
\qbezier(110,-10)(115,-20)(120,-20)

\qbezier(100,-20)(105,-20),(109,-12)
\qbezier(111,-8)(115,0),(120,0)

\qbezier(120,0)(125,0)(130,-10)
\qbezier(130,-10)(135,-20)(140,-20)

\qbezier(120,-20)(125,-20),(129,-12)
\qbezier(131,-8)(135,0),(140,0)

\put(142,-20){\vector(1,0){0}}

\put(142,0){\vector(1,0){0}}

\put(145,0){\small{$x_1$}}

\put(145,-20){\small{$x_2$}}

\end{picture} 
}
\]

\caption{A diagram of $\mathsf{b}^{2k+2}$}\label{braid-2k+2-fig}

\end{figure}
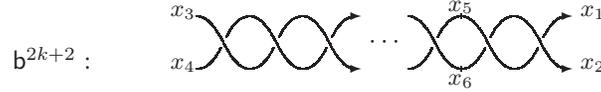

We put markings on $\mathsf{b}^{2k+2}$ as in Figure \ref{braid-2k+2-fig}. Let $\Gamma_{00}$, $\Gamma_{01}$, $\Gamma_{10}$ and $\Gamma_{11}$ be the MOY graphs in Figure \ref{chain-b+2-fig}. By the induction hypothesis and Lemma \ref{lemma-b+2}, $C(\mathsf{b}^{2k+2})=C(\mathsf{b}^{2k}) \otimes_{\C[x_5,x_6]} C(\mathsf{b}^{2})$ is homotopic to the total complex of the following double complex \eqref{complex-b+2k+2}\footnote{To simplify the notations, we omit the $q$-grading shifts applied to the rows of double complex \eqref{complex-b+2k+2} and will do the same in the rest of this proof.}.
\begin{equation}\label{complex-b+2k+2}
\xymatrix{
0 \ar[r]& C(\Gamma_{01})\{q^3\} \ar[rr]^{x_1-x_5} && C(\Gamma_{01})\{q\} \ar[rr]^{\chi_r^1} && C(\Gamma_{00}) \ar[r] & 0 \\
0 \ar[r]& C(\Gamma_{11})\{q^3\} \ar[rr]^{x_1-x_5} \ar[u]^{\chi_l^1} && C(\Gamma_{11})\{q\} \ar[rr]^{\chi_r^1} \ar[u]^{-\chi_l^1} && C(\Gamma_{10}) \ar[r] \ar[u]^{\chi_l^1} & 0 \\
 & \vdots \ar[u]^{x_5-x_3} && \vdots \ar[u]^{-(x_5-x_3)} && \vdots \ar[u]^{x_5-x_3} &  \\
0 \ar[r]& C(\Gamma_{11})\{q^3\} \ar[rr]^{x_1-x_5} \ar[u]^{x_5-x_4} && C(\Gamma_{11})\{q\} \ar[rr]^{\chi_r^1} \ar[u]^{-(x_5-x_4)} && C(\Gamma_{10}) \ar[r] \ar[u]^{x_5-x_4} & 0 \\
0 \ar[r]& C(\Gamma_{11})\{q^3\} \ar[rr]^{x_1-x_5} \ar[u]^{x_5-x_3} && C(\Gamma_{11})\{q\} \ar[rr]^{\chi_r^1} \ar[u]^{-(x_5-x_3)} && C(\Gamma_{10}) \ar[r] \ar[u]^{x_5-x_3} & 0
}
\end{equation}

Note that, except the top row, all the other rows in \eqref{complex-b+2k+2} are of the form 
\begin{equation}\label{complex-b+2k+2-one-row}
\xymatrix{
0 \ar[r]& C(\Gamma_{11})\{q^3\} \ar[rr]^{x_1-x_5}  && C(\Gamma_{11})\{q\} \ar[rr]^{\chi_r^1}  && C(\Gamma_{10}) \ar[r]  & 0.
}
\end{equation}
In the chain complex \eqref{complex-b+2k+2-one-row}, we replace the first $C(\Gamma_{11})$ with $\left.%
\begin{array}{c}
  C(\Gamma_{10})\{q\} \\
  \oplus \\
  C(\Gamma'')\{q^{-1}\} 
\end{array}%
\right.$ using the homotopy equivalence \eqref{homotopy-inverses-1} and the second $C(\Gamma_{11})$ with $\left.%
\begin{array}{c}
  C(\Gamma'')\{q\} \\
  \oplus \\
  C(\Gamma_{10})\{q^{-1}\} 
\end{array}%
\right.$ using the homotopy equivalence \eqref{homotopy-inverses-2}. This means that complex \ref{complex-b+2k+2-one-row} is isomorphic to 

\begin{equation}\label{complex-b+2k+2-one-row-2}
\xymatrix{
0 \ar[r]& {\left.%
\begin{array}{c}
  C(\Gamma_{10})\{q^4\} \\
  \oplus \\
  C(\Gamma'')\{q^{2}\} 
\end{array}%
\right.} \ar[rr]^{d}  && {\left.%
\begin{array}{c}
  C(\Gamma'')\{q^2\} \\
  \oplus \\
  C(\Gamma_{10})
\end{array}%
\right.} \ar[rr]^>>>>>>>>>>{(\chi_r^1 \circ \xi_r^0 , ~\chi_r^1 \circ J)}  && C(\Gamma_{10}) \ar[r]  & 0,
}
\end{equation}
where 
\[
d = \left(%
\begin{array}{cc}
  P \circ \mathtt{m}(x_1-x_5) \circ \chi_r^0 & P \circ \mathtt{m}(x_1-x_5) \circ J \\
  \chi_r^1 \circ \mathtt{m}(x_1-x_5) \circ \chi_r^0 & \chi_r^1 \circ \mathtt{m}(x_1-x_5) \circ J
\end{array}%
\right).
\]
By Lemma \ref{lemma-J-P-chi-xi-compose}, $C(\Gamma_{10}) \xrightarrow{\chi_r^1 \circ J} C(\Gamma_{10})$ is a homotopy equivalence. So, by Lemma \ref{gaussian-elimination}, complex \eqref{complex-b+2k+2-one-row-2} is homotopic to 
\begin{equation}\label{complex-b+2k+2-one-row-3}
\xymatrix{
0 \ar[r]& {\left.%
\begin{array}{c}
  C(\Gamma_{10})\{q^4\} \\
  \oplus \\
  C(\Gamma'')\{q^{2}\} 
\end{array}%
\right.} & \ar[rrrr]^{(P \circ \mathtt{m}(x_1-x_5) \circ \chi_r^0,~ P \circ \mathtt{m}(x_1-x_5) \circ J)}  &&&& &
  C(\Gamma'')\{q^2\} \ar[r]  & 0,
}
\end{equation}
Note that $P$ and $J$ commute with $\mathtt{m}(x_1)$. By homotopy \eqref{comp-J-P-x-5} and \eqref{comp-J-P-1}, it is easy to see that $P \circ \mathtt{m}(x_1-x_5) \circ J$ is a homotopy equivalence. Using Lemma \ref{gaussian-elimination} again, we get that complex \eqref{complex-b+2k+2-one-row-3} is homotopic to 
\[
\xymatrix{
0 \ar[r]& C(\Gamma_{10})\{q^4\} \ar[r] & 0}.
\]

The above shows that $C(\mathsf{b}^{2k+2})$ is homotopic to the chain complex
\begin{equation}\label{complex-b+2k+2-reduced-1}
\xymatrix{
 C(\Gamma_{01})\{q^3\} \ar[rr]^{x_1-x_5} && C(\Gamma_{01})\{q\} \ar[rr]^{\chi_r^1} && C(\Gamma_{00}) \ar[r] & 0 \\
 C(\Gamma_{10})\{q^4\}  \ar[u]^{\chi_l^1\circ\chi_r^0}  \\
 \vdots \ar[u]^{P \circ\mathtt{m}(x_5-x_3) \circ\chi_r^0}  \\
 C(\Gamma_{10})\{q^4\}  \ar[u]^{P \circ\mathtt{m}(x_5-x_4) \circ\chi_r^0}  \\
C(\Gamma_{10})\{q^4\}  \ar[u]^{P \circ\mathtt{m}(x_5-x_3) \circ\chi_r^0} \\
0 \ar[u]
}
\end{equation}

Note that 
\begin{itemize}
	\item $\chi_r^0$ commutes with $\mathtt{m}(x_3)$, $\mathtt{m}(x_4)$ and $\mathtt{m}(x_5)$ as endomorphisms of $C(\Gamma_{10})$,
	\item by Lemma \ref{lemma-J-P-chi-xi-compose}, $P \circ \chi_r^0 \simeq c \cdot \id_{C(\Gamma_{10})}$ for some non-zero scalar $c \in \C$,
	\item $\mathtt{m}(x_1) \simeq \mathtt{m}(x_5)$ as endomorphisms of $C(\Gamma_{10})$.
\end{itemize}
Thus, as endomorphisms of $C(\Gamma_{10})$,
\begin{equation}\label{homotopy-b+2k+2-1}
P \circ\mathtt{m}(x_5-x_3) \circ\chi_r^0 \simeq c\cdot \mathtt{m}(x_1-x_3), \hspace{1pc} P \circ\mathtt{m}(x_5-x_4) \circ\chi_r^0 \simeq c\cdot \mathtt{m}(x_1-x_4).
\end{equation}
Moreover, note that $\mathtt{m}(x_3) \simeq \mathtt{m}(x_5)$ as endomorphisms of $C(\Gamma_{01})$. So, as endomorphisms of $C(\Gamma_{01})$,
\begin{equation}\label{homotopy-b+2k+2-1-extra}
\mathtt{m}(x_1-x_5) \simeq \mathtt{m}(x_1-x_3).
\end{equation}

\begin{equation}\label{complex-b+2k+2-diagram-1}
\xymatrix{
C(\Gamma_{10}) \ar[rr]^{\chi_r^0} \ar[d]^{\chi_l^1} && C(\Gamma_{11}) \ar[d]^{\chi_l^1}  \\
C(\Gamma_{00}) \ar[rr]^{\chi_r^0} && C(\Gamma_{01})
}
\end{equation}

Now consider $C(\Gamma_{10}) \xrightarrow{\chi_l^1\circ\chi_r^0} C(\Gamma_{01})$. The homomorphisms $\chi_l^1$ and $\chi_r^0$ act on disjoint parts of the MOY graph. So diagram \eqref{complex-b+2k+2-diagram-1} commutes. Let $\Gamma'$ be the MOY graph given in Figure \ref{MOY-decomp-II-fig}, under the identification $C(\Gamma') \simeq C(\Gamma_{10}) \simeq C(\Gamma_{01})$, $C(\Gamma_{10}) \xrightarrow{\chi_l^1} C(\Gamma_{00})$ becomes $C(\Gamma') \xrightarrow{\chi^1} C(\Gamma_{00})$ and $C(\Gamma_{00}) \xrightarrow{\chi_r^0} C(\Gamma_{01})$ becomes $C(\Gamma_{00}) \xrightarrow{\chi^0} C(\Gamma')$. Recall that $\chi^0 \circ \chi^1 = (x_1-x_4) \cdot \id_{C(\Gamma')}$. Putting these together, we get
\begin{equation}\label{homotopy-b+2k+2-2}
\chi_l^1\circ\chi_r^0 \simeq \chi_r^0\circ\chi_l^1 \simeq (x_1-x_4) \cdot \id_{C(\Gamma')}.
\end{equation}

Using homotopy \eqref{homotopy-b+2k+2-1}, \eqref{homotopy-b+2k+2-1-extra} and \eqref{homotopy-b+2k+2-2}, it is easy to see that chain complex \eqref{complex-b+2k+2-reduced-1} is isomorphic to $\mathscr{B}_{k+1}$ and, therefore, $C(\mathsf{b}^{2k+2}) \simeq \mathscr{B}_{k+1}$. This completes the induction.
\end{proof}

\section{The Rasmussen Invariants of $2$-Cables}\label{sec-ras}

In this section, we prove Theorems \ref{thm-linearity-general} and \ref{thm-linearity-2cable}. We start by reviewing the structure of the equivariant Khovanov-Rozansky homology $H_{x^{N+1}-ax}$.

\subsection{Structure of $H_{x^{N+1}-ax}$} In \cite{Lobb-gornik}, Lobb determined the filtration on $H_{x^{N+1}-x}$. Using this, he described in \cite{Lobb-2-twists} the structure of the free part of $H_{x^{N+1}-ax}$, which we will use in our proofs of Theorems \ref{thm-linearity-general} and \ref{thm-linearity-2cable}. Since Lobb's description in \cite{Lobb-2-twists} contains only very brief explanations, we include a slightly more detailed proof here for the convenience of the reader.

A graded and free module does not necessarily admit a homogeneous basis. We call a module graded-free if it is graded and free and admits a homogeneous basis.

\begin{lemma}\label{lemma-graph-homology-free}
Let $\Gamma$ be an embedded MOY graph in the plane. Here, ``embedded" means $\Gamma$ contains no crossings or virtual crossings. Then $H_{x^{N+1}-ax}(\Gamma)$ is a finitely generated graded-free $\C[a]$-module.
\end{lemma}

\begin{proof}
For $\ve\in \zed_2$, denote by $H_{x^{N+1}-ax}^\ve(\Gamma)$ the component of $H_{x^{N+1}-ax}(\Gamma)$ of $\zed_2$-grading $\ve$. If we replace all wide edges in $\Gamma$ by a pair of parallel arcs, that is, replacing $\Gamma_1$ in Figure \ref{resolutions-fig} by $\Gamma_0$ in Figure \ref{resolutions-fig}, then we get a finite collection of embedded circles in the plane. Let $\ve$ be the element of $\zed_2$ that is congruent to the cardinality of this collection modulo $2$. 

Let $x_1,\dots,x_m$ be the variables associated to all marked points of $\Gamma$. Then the $q$-grading of $C_{x^{N+1}-ax}(\Gamma)$ comes from the grading of $\C[a,x_1,\dots,x_m]$ given by $\deg a = 2N$ and $\deg x_i = 2$ for $i=1,\dots,m$. The grading of $\C[x_1,\dots,x_m]$ induces a filtration $\fil_x$ on $C_{x^{N+1}-ax}(\Gamma)$. Using this filtration, we can repeat the argument in \cite[Proof of Proposition 2.19]{Wu7} and get that
\begin{eqnarray}
\label{eq-graph-homology-free-1} H_{x^{N+1}-ax}^{\ve+1}(\Gamma) & = & 0, \\
\label{eq-graph-homology-free-2} \fil^k_x H_{x^{N+1}-ax}^\ve(\Gamma) / \fil^{k-1}_x H_{x^{N+1}-ax}^\ve(\Gamma) & \cong & H_{x^{N+1}}^{\ve,k}(\Gamma) \otimes_\C \C[a],
\end{eqnarray}
where $H_{x^{N+1}}^{\ve,k}(\Gamma)$ is the component of $H_{x^{N+1}}(\Gamma)$ with $\zed_2$-grading $\ve$ and $q$-grading $k$. Note that the right-hand side of \eqref{eq-graph-homology-free-2} is a finitely generated free $\C[a]$-module and that the filtration $\fil_x$ of $H_{x^{N+1}-ax}^\ve(\Gamma)$ is bounded. So it is easy to prove by an induction on $k$ that $H_{x^{N+1}-ax}^\ve(\Gamma)$ is a finitely generated free $\C[a]$-module. Recall that $H_{x^{N+1}-ax}^\ve(\Gamma)$ is a graded $\C[a]$-module and its grading is bounded below. So $H_{x^{N+1}-ax}^\ve(\Gamma)$ is a finitely generated graded-free $\C[a]$-module by \cite[Lemma 3.3]{Wu-color}.
\end{proof}

\begin{lemma}\label{lemma-free-chain-complex}
Assume that $C^\ast$ is a bounded chain complex of finitely generated graded-free $\C[a]$-modules and its differential map preserves the grading. Denote by $\hat{C}^\ast$ the filtered chain complex $\hat{C}^\ast = C^\ast/(a-1)C^\ast$ of finite dimensional $\C$-linear spaces. Denote by $H^\ast$ and $\hat{H}^\ast$ the homologies of $C^\ast$ and $\hat{C}^\ast$. Then
\[
H^n \cong \mathscr{F}^n \oplus \mathscr{T}^n,
\]
where $\mathscr{F}^n$ is a graded-free $\C[a]$-module whose graded rank is equal to the filtered dimension of $\hat{H}^n$ and $\mathscr{T}^n$ is either $0$ or a finitely generated graded torsion $\C[a]$-module.
\end{lemma}
\begin{proof}
(Following \cite{Lobb-2-twists}.) Choose a homogeneous basis for each $C^n$ and express the differential map of $C^\ast$ as matrices. Then each entry of these matrices is of the form $c\cdot a^m$, where $c \in \C$. Using this, it is not hard to modify the homogeneous basis to decompose $C^\ast$ into a finite direct sum of graded chain complexes of the types
\begin{enumerate}
	\item $0\rightarrow \C[a]\{q^m\} \rightarrow 0$,
	\item $0\rightarrow \C[a]\{q^m\} \xrightarrow{a^k} \C[a]\{q^{m-k}\} \rightarrow 0$.
\end{enumerate}
A type (1) chain complex contributes a rank one free component of graded rank $q^m$ to $H^\ast$ and a $1$-dimensional subspace of filtered dimension $q^m$ to $\hat{H}^\ast$. A type (2) chain complex contributes a torsion component to $H^\ast$ and nothing to $\hat{H}^\ast$. The lemma follows from these observations.
\end{proof}

The $\zed_2$-grading of the Khovanov-Rozansky homology of a knot is trivial. We do not keep track of this grading in the rest of this section.

\begin{corollary}\cite{Lobb-2-twists}\label{cor-H-equi-structure}
For any knot $K$, denote by $H_{x^{N+1}-ax}^m(K)$ the component of $H_{x^{N+1}-ax}(K)$ of homological grading $m$. Then
\[
H_{x^{N+1}-ax}^m(K) \cong 
\begin{cases}
\mathscr{T}^0 \oplus \bigoplus_{l=1}^{N} \C[a]\{q^{N+1-2l+s_N(K)}\} & \text{if } m=0, \\
\mathscr{T}^m & \text{if } m \neq 0,
\end{cases}
\]
where $s_N(K)$ is the $\mathfrak{sl}(N)$ Rasmussen invariant of $K$ and, for each $m\in \zed$, $\mathscr{T}^m$ is either $0$ or a finitely generated graded torsion $\C[a]$-module.
\end{corollary}

\begin{proof}
(Following \cite{Lobb-2-twists}.) In \cite[Theorem 2]{Gornik}, Gornik proved that that $H_{x^{N+1}-x}^m(K)=0$ if $m\neq 0$. In \cite{Lobb-gornik}, Lobb proved that the filtered dimension of $H_{x^{N+1}-x}^0(K)$ is $\sum_{l=1}^{N} q^{N+1-2l+s_N(K)}$. By Lemmas \ref{lemma-graph-homology-free} and \ref{lemma-free-chain-complex}, these imply the corollary.
\end{proof}

\subsection{$2$-braids in a knot diagram} We prove Theorem \ref{thm-linearity-general} in this subsection. We first recall Gornik's computation of $H_{x^{N+1}-x}$ in \cite{Gornik}. 

\begin{figure}[ht]
\[
\xymatrix{
\text{crossing:} & \setlength{\unitlength}{1pt}
\begin{picture}(40,20)(-10,-20)

\put(-10,0){\small{$e_3$}}

\put(-10,-20){\small{$e_4$}}

\qbezier(0,0)(5,0)(10,-10)
\qbezier(10,-10)(15,-20)(20,-20)

\qbezier(0,-20)(5,-20),(9,-12)
\qbezier(11,-8)(15,0),(20,0)

\put(22,-20){\vector(1,0){0}}

\put(22,0){\vector(1,0){0}}

\put(25,0){\small{$e_1$}}

\put(25,-20){\small{$e_2$}}

\end{picture} & \text{or} & \setlength{\unitlength}{1pt}
\begin{picture}(40,20)(-10,0)

\put(-10,20){\small{$e_3$}}

\put(-10,0){\small{$e_4$}}

\qbezier(0,0)(5,0)(10,10)
\qbezier(10,10)(15,20)(20,20)

\qbezier(0,20)(5,20),(9,12)
\qbezier(11,8)(15,0),(20,0)

\put(22,20){\vector(1,0){0}}

\put(22,0){\vector(1,0){0}}

\put(25,20){\small{$e_1$}}

\put(25,0){\small{$e_2$}}

\end{picture} \\
\text{wide edge:} & \setlength{\unitlength}{1pt}
\begin{picture}(50,20)(-10,0)

\put(-10,20){\small{$e_3$}}

\put(-10,0){\small{$e_4$}}

\put(0,0){\vector(1,1){10}}

\put(0,20){\vector(1,-1){10}}

\put(20,10){\vector(1,1){10}}

\put(20,10){\vector(1,-1){10}}

\put(35,20){\small{$e_1$}}

\put(35,0){\small{$e_2$}}

\linethickness{3pt}

\put(10,10){\line(1,0){10}}

\end{picture}
}
\]

\caption{Labelling around a crossing or a wide edge}\label{gornik-state-fig}

\end{figure}
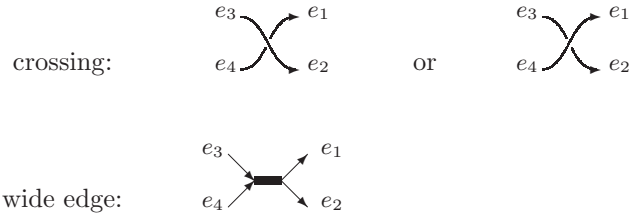

We view a knotted MOY graph as an oriented $4$-valent graph embedded in the plane with two types of vertices: crossings and wide edges. We name the edge around a crossing or a wide edge as in Figure \ref{gornik-state-fig}. Let $\Sigma_N= \{e^{\frac{2k\pi i}{N}}~|~ k=0,\dots,N-1\}$. A state of the knotted MOY graph $\Gamma$ is a function $\varphi$ from the set of edges of $\Gamma$ to $\Sigma_N$ such that 
\begin{itemize}
	\item at each crossing of $\Gamma$, $\varphi(e_1) = \varphi(e_4)$ and $\varphi(e_2) = \varphi(e_3)$,
	\item at each wide edge of $\Gamma$, $\varphi(e_1) \neq \varphi(e_2)$ and $\{\varphi(e_1),~ \varphi(e_2)\} = \{\varphi(e_3),~ \varphi(e_4)\}$.
\end{itemize}
For each state $\varphi$ of $\Gamma$, there is a non-zero element $\mathbf{a}_\varphi$ of $H_{x^{N+1}-x}(\Gamma)$. Denote by $\mathcal{S}_N (\Gamma)$ the set of all states of $\Gamma$. Then the set $\{\mathbf{a}_\varphi~|~ \varphi \in \mathcal{S}_N (\Gamma)\}$ is a $\C$-linear basis for $H_{x^{N+1}-x}(\Gamma)$. Moreover, if $x$ is a variable assigned to a marked point on an edge $e$ of $\Gamma$, then 
\begin{equation}\label{eq-Gornik-basis-action}
x\cdot\mathbf{a}_\varphi = \varphi(e)\cdot\mathbf{a}_\varphi.
\end{equation}

Now we are ready to prove Theorem \ref{thm-linearity-general}.

\begin{proof}[Proof of Theorem \ref{thm-linearity-general}]
Let $D_0$, $D_k$ and $\Gamma$ be as in Subsection \ref{subsec-exact} and assume that $D_0$ is a knot diagram. For $i=1,2,3,4$, denote by $e_i$ the edge in $\Gamma$ marked by $x_i$. In order for $D_0$ to be a knot diagram, we know that $e_1$ must be connected to $e_4$ via the part of $\Gamma$ not shown in Figure \ref{D-k-Gamma-fig}. This means that, for any state $\varphi$ of $\Gamma$, $\varphi(e_1)=\varphi(e_4)\neq \varphi(e_3)$. Thus, by equation \eqref{eq-Gornik-basis-action}, $(x_1-x_3)\cdot\mathbf{a}_\varphi = (\varphi(e_1)-\varphi(e_3)) \cdot\mathbf{a}_\varphi \neq 0$. This implies that $H_{x^{N+1}-x}(\Gamma) \xrightarrow{x_1-x_3} H_{x^{N+1}-x}(\Gamma)\{q^{-2}\}$ is an isomorphism. Then long exact sequence \eqref{minor-long-exact-sequence} tells us that $\mathscr{H}_{x^{N+1}-x}(\Gamma)=0$. By Lemma \ref{lemma-graph-homology-free}, $\mathscr{H}_{x^{N+1}-ax}(\Gamma)$ is computed by a chain complex of finitely generated graded-free $\C[a]$-modules. Then, from Lemma \ref{lemma-free-chain-complex}, we know that $\mathscr{H}_{x^{N+1}-ax}(\Gamma)$ is purely torsion.
{\tiny
\[
\cdots \rightarrow \mathscr{H}_{x^{N+1}-ax}^{2k-1}(\Gamma)\{q^{2k(N+1)-1}\} \rightarrow H_{x^{N+1}-ax}^{0}(D_{k-1})\{q^{2(N-1)}\} \rightarrow H_{x^{N+1}-ax}^{0}(D_{k}) \rightarrow \mathscr{H}_{x^{N+1}-ax}^{2k}(\Gamma)\{q^{2k(N+1)-1}\} \rightarrow \cdots
\]}
Now consider the above segment of long exact sequence \eqref{long-exact-sequence} Assume $k\geq \frac{c_-+2}{2}$. Then, by the construction of the Khovanov-Rozansky chain complex, we know that $C^l(\Gamma)=0$ for $l \geq 2k-1$. Using long exact sequence \eqref{minor-long-exact-sequence}, this implies that $\mathscr{H}^{2k}(\Gamma)=0$. So, in the above exact sequence, the homomorphism
\begin{equation}\label{eq-hom-inj-surj}
H_{x^{N+1}-ax}^{0}(D_{k-1})\{q^{2(N-1)}\} \rightarrow H_{x^{N+1}-ax}^{0}(D_{k})
\end{equation}
is surjective. Recall that $\mathscr{H}_{x^{N+1}-ax}(\Gamma)$ is purely torsion. So the restriction of homomorphism \eqref{eq-hom-inj-surj} onto the free part of $H_{x^{N+1}-ax}^{0}(D_{k-1})\{q^{2(N-1)}\}$ is injective. Thus, homomorphism \eqref{eq-hom-inj-surj} induces an isomorphism of the free parts of $H_{x^{N+1}-ax}^{0}(D_{k-1})\{q^{2(N-1)}\}$ and $H_{x^{N+1}-ax}^{0}(D_{k})$. By Corollary \ref{cor-H-equi-structure}, this implies that
\begin{equation}\label{eq-linearity-general+}
s_N(D_k) = s_N(D_{k-1}) + 2(N-1) \text{ if } k\geq \frac{c_- +2}{2}.
\end{equation}
Since $s_N$ is a concordance invariant, we know that $s_N(\overline{K})=-s_N(K)$ for any knot $K$, where $\overline{K}$ is the mirror image of $K$. Using this and equation \eqref{eq-linearity-general+}, it is easy to show that
\begin{equation}\label{eq-linearity-general-}
s_N(D_k) = s_N(D_{k-1}) + 2(N-1) \text{ if } k \leq \frac{-c_+ -2}{2}.
\end{equation}
\end{proof}

\subsection{$(2,2k+1)$ cables} Next, we improve the proof of Theorem \ref{thm-linearity-general} for $(2,2k+1)$ cables and prove Theorem \ref{thm-linearity-2cable}. For this purpose, we briefly review some basic properties of the colored $\mathfrak{sl}(N)$ homology in \cite{Wu-color,Wu-color-equi,Wu-color-ras}. 

\begin{figure}[ht]
\[
\xymatrix{
\setlength{\unitlength}{1pt}
\begin{picture}(70,40)(10,-10)

\qbezier(50,10)(50,20)(62,20)

\qbezier(50,10)(50,0)(62,0)

\put(68,20){\line(1,0){12}}

\put(68,0){\line(1,0){12}}

\put(65,-10){\line(0,1){40}}

\linethickness{3pt}

\put(10,10){\line(1,0){40}}

\end{picture}   & \ar@{<~>}[r] && \setlength{\unitlength}{1pt}
\begin{picture}(70,40)(10,-10)

\qbezier(50,10)(50,20)(62,20)

\qbezier(50,10)(50,0)(62,0)

\put(62,20){\line(1,0){18}}

\put(62,0){\line(1,0){18}}

\put(30,-10){\line(0,1){40}}

\linethickness{3pt}

\put(10,10){\line(1,0){16}}

\put(34,10){\line(1,0){16}}

\end{picture}  \\
\setlength{\unitlength}{1pt}
\begin{picture}(70,40)(10,-10)

\qbezier(50,10)(50,20)(62,20)

\qbezier(50,10)(50,0)(62,0)

\put(68,20){\line(1,0){12}}

\put(68,0){\line(1,0){12}}

\linethickness{3pt}

\put(65,-10){\line(0,1){40}}

\put(10,10){\line(1,0){40}}

\end{picture} & \ar@{<~>}[r] &&  \setlength{\unitlength}{1pt}
\begin{picture}(70,40)(10,-10)

\qbezier(50,10)(50,20)(62,20)

\qbezier(50,10)(50,0)(62,0)

\put(62,20){\line(1,0){18}}

\put(62,0){\line(1,0){18}}

\linethickness{3pt}

\put(30,-10){\line(0,1){40}}

\put(10,10){\line(1,0){16}}

\put(34,10){\line(1,0){16}}

\end{picture} \\
\setlength{\unitlength}{1pt}
\begin{picture}(70,20)(90,-20)

\qbezier(130,-10)(135,-20)(140,-20)

\qbezier(130,-10)(135,0),(140,0)

\qbezier(140,0)(145,0)(150,-10)
\qbezier(150,-10)(155,-20)(160,-20)

\qbezier(140,-20)(145,-20),(149,-12)
\qbezier(151,-8)(155,0),(160,0)

\linethickness{3pt}

\put(90,-10){\line(1,0){40}}

\end{picture} & \ar@{<~>}[r]  && \setlength{\unitlength}{1pt}
\begin{picture}(70,40)(10,-10)

\qbezier(50,10)(50,20)(62,20)

\qbezier(50,10)(50,0)(62,0)

\put(62,20){\line(1,0){18}}

\put(62,0){\line(1,0){18}}

\linethickness{3pt}

\put(10,10){\line(1,0){40}}

\end{picture}
}
\]

\caption{Additional invariant moves up to overall grading shifts of the colored $\mathfrak{sl}(N)$ homology}\label{fork-sliding-fig} 

\end{figure}
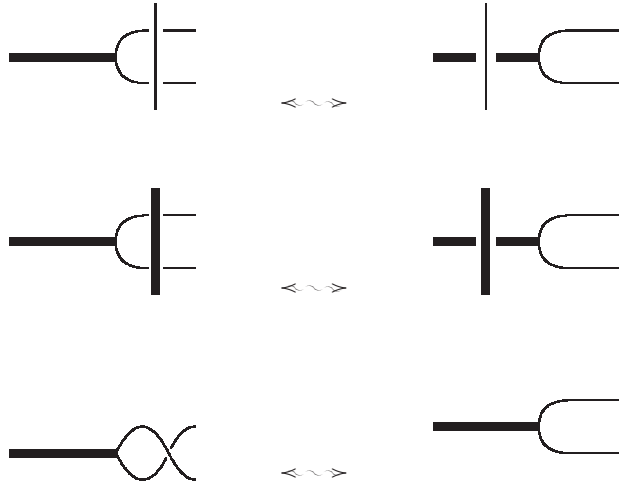

In the terminology of the colored $\mathfrak{sl}(N)$ homology, a wide edge in an MOY graph is an edge colored by $2$ and a thin edge is an edge colored by $1$. Each wide edge is equipped with the orientation pointing from the end with two incoming thin edges to the end with two outgoing thin edges. This way, at every vertex, the total color of inward edges and the total color of the outward edges are both $2$. In the colored $\mathfrak{sl}(N)$ homology, we allow crossings between a wide edge and a thin edge and crossings between two wide edges. The colored $\mathfrak{sl}(N)$ homology of a knotted MOY graph is invariant under all Reidemeister moves, including those involving wide edges. Up to overall grading shifts, it is also invariant under the moves in Figure \ref{fork-sliding-fig}, where, in the first two moves, we allow the vertical edge to be behind the ``fork" too. See \cite[Theorem 12.1 and Lemma 13.6]{Wu-color} and \cite[Theorem 7.1 and Lemma 8.4]{Wu-color-equi} for full statements.

We are now ready to prove Theorem \ref{thm-linearity-2cable}.

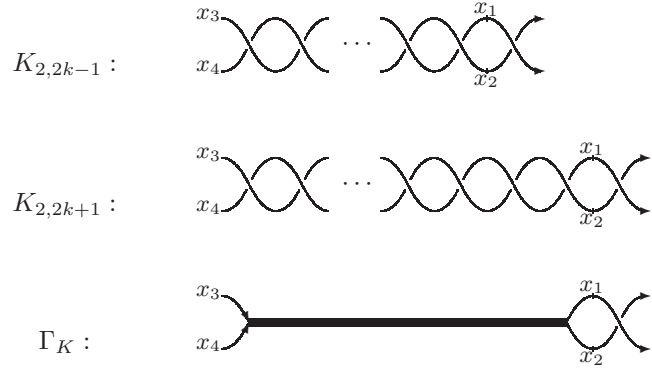
\begin{figure}[ht]
\[
\xymatrix{
K_{2,2k-1}: & \setlength{\unitlength}{1pt}
\begin{picture}(170,20)(-10,-20)

\put(-10,0){\small{$x_3$}}

\put(-10,-20){\small{$x_4$}}

\qbezier(0,0)(5,0)(10,-10)
\qbezier(10,-10)(15,-20)(20,-20)

\qbezier(0,-20)(5,-20),(9,-12)
\qbezier(11,-8)(15,0),(20,0)

\qbezier(20,0)(25,0)(30,-10)
\qbezier(30,-10)(35,-20)(40,-20)

\qbezier(20,-20)(25,-20),(29,-12)
\qbezier(31,-8)(35,0),(40,0)

\put(45,-12){$\cdots$}

\qbezier(60,0)(65,0)(70,-10)
\qbezier(70,-10)(75,-20)(80,-20)

\qbezier(60,-20)(65,-20),(69,-12)
\qbezier(71,-8)(75,0),(80,0)

\qbezier(80,0)(85,0)(90,-10)
\qbezier(90,-10)(95,-20)(100,-20)

\qbezier(80,-20)(85,-20),(89,-12)
\qbezier(91,-8)(95,0),(100,0)

\qbezier(100,0)(105,0)(110,-10)
\qbezier(110,-10)(115,-20)(120,-20)

\qbezier(100,-20)(105,-20),(109,-12)
\qbezier(111,-8)(115,0),(120,0)

\put(122,-20){\vector(1,0){0}}

\put(122,0){\vector(1,0){0}}

\put(95,2){\small{$x_1$}}

\put(95,-25){\small{$x_2$}}

\put(100,-1){\line(0,1){2}}

\put(100,-21){\line(0,1){2}}

\end{picture}  \\
K_{2,2k+1}: & \setlength{\unitlength}{1pt}
\begin{picture}(170,20)(-10,-20)

\put(-10,0){\small{$x_3$}}

\put(-10,-20){\small{$x_4$}}

\qbezier(0,0)(5,0)(10,-10)
\qbezier(10,-10)(15,-20)(20,-20)

\qbezier(0,-20)(5,-20),(9,-12)
\qbezier(11,-8)(15,0),(20,0)

\qbezier(20,0)(25,0)(30,-10)
\qbezier(30,-10)(35,-20)(40,-20)

\qbezier(20,-20)(25,-20),(29,-12)
\qbezier(31,-8)(35,0),(40,0)

\put(45,-12){$\cdots$}

\qbezier(60,0)(65,0)(70,-10)
\qbezier(70,-10)(75,-20)(80,-20)

\qbezier(60,-20)(65,-20),(69,-12)
\qbezier(71,-8)(75,0),(80,0)

\qbezier(80,0)(85,0)(90,-10)
\qbezier(90,-10)(95,-20)(100,-20)

\qbezier(80,-20)(85,-20),(89,-12)
\qbezier(91,-8)(95,0),(100,0)

\qbezier(100,0)(105,0)(110,-10)
\qbezier(110,-10)(115,-20)(120,-20)

\qbezier(100,-20)(105,-20),(109,-12)
\qbezier(111,-8)(115,0),(120,0)

\put(135,2){\small{$x_1$}}

\put(135,-25){\small{$x_2$}}

\put(140,-1){\line(0,1){2}}

\put(140,-21){\line(0,1){2}}

\qbezier(120,0)(125,0)(130,-10)
\qbezier(130,-10)(135,-20)(140,-20)

\qbezier(120,-20)(125,-20),(129,-12)
\qbezier(131,-8)(135,0),(140,0)

\qbezier(140,0)(145,0)(150,-10)
\qbezier(150,-10)(155,-20)(160,-20)

\qbezier(140,-20)(145,-20),(149,-12)
\qbezier(151,-8)(155,0),(160,0)

\put(162,-20){\vector(1,0){0}}

\put(162,0){\vector(1,0){0}}

\end{picture} \\
\Gamma_K: & \setlength{\unitlength}{1pt}
\begin{picture}(170,20)(-10,-20)

\put(-10,0){\small{$x_3$}}

\put(-10,-20){\small{$x_4$}}

\qbezier(0,0)(5,0)(10,-10)

\qbezier(0,-20)(5,-20),(10,-10)

\put(10,-10){\vector(1,-2){0}}

\put(10,-10){\vector(1,2){0}}

\put(135,2){\small{$x_1$}}

\put(135,-25){\small{$x_2$}}

\put(140,-1){\line(0,1){2}}

\put(140,-21){\line(0,1){2}}

\qbezier(130,-10)(135,-20)(140,-20)

\qbezier(130,-10)(135,0),(140,0)

\qbezier(140,0)(145,0)(150,-10)
\qbezier(150,-10)(155,-20)(160,-20)

\qbezier(140,-20)(145,-20),(149,-12)
\qbezier(151,-8)(155,0),(160,0)

\put(162,-20){\vector(1,0){0}}

\put(162,0){\vector(1,0){0}}

\linethickness{3pt}

\put(10,-10){\line(1,0){120}}

\end{picture}
}
\]

\caption{Local diagrams of $K_{2,2k-1}$, $K_{2,2k+1}$ and $\Gamma_K$}\label{cables-fig}

\end{figure}

\begin{proof}[Proof of Theorem \ref{thm-linearity-2cable}]
For simplicity, we fix a diagram of $K$ with writhe $0$. (The number of positive/negative crossings in this diagram is most likely different from $c_\pm$.) The blackboard framing of $K$ is the Seifert framing of $K$. Pushing $K$ slightly in the direction of the blackboard framing, we get a knot $K'$. Then $K \cup K'= K_{2,0}$. Choose a small arc in $K$ away from crossings. Its image in $K'$ is an arc parallel to it. Replacing this pair of parallel arcs in $K \cup K'$ by $\mathsf{b}^{2k\pm 1}$, we get $K_{2,2k\pm 1}$. Also, if we replace this pair of parallel arcs by a wide edge followed by a positive half twist, then we get a knotted MOY graph $\Gamma_K$. (See Figure \ref{cables-fig}.) Applying Theorem \ref{thm-exact-sequence} to the triple $(K_{2,2k-1}, K_{2,2k+1},\Gamma_K)$, we get a long exact sequence  
{\tiny
\[
\cdots \rightarrow \mathscr{H}_{x^{N+1}-ax}^{2k-1}(\Gamma_K)\{q^{2k(N+1)-1}\} \rightarrow H_{x^{N+1}-ax}^{0}(K_{2,2k-1})\{q^{2(N-1)}\} \rightarrow H_{x^{N+1}-ax}^{0}(K_{2,2k+1}) \rightarrow \mathscr{H}_{x^{N+1}-ax}^{2k}(\Gamma_K)\{q^{2k(N+1)-1}\} \rightarrow \cdots
\]}

By the argument used in the proof of Theorem \ref{thm-linearity-general}, we know that $\mathscr{H}_{x^{N+1}-ax}(\Gamma_K)$ is purely torsion. So the homomorphism 
\begin{equation}\label{eq-homomorphism-2cable}
H_{x^{N+1}-ax}^{0}(K_{2,2k-1})\{q^{2(N-1)}\} \rightarrow H_{x^{N+1}-ax}^{0}(K_{2,2k+1})
\end{equation}
in the above exact sequence is injective on the free part of $H_{x^{N+1}-ax}^{0}(K_{2,2k-1})\{q^{2(N-1)}\}$. 

Recall that the writhe of $K$ is $0$. Using the moves in Figure \ref{fork-sliding-fig} and Proposition \ref{MOY-decomp-II} and keeping track of the grading shifts in \cite[Definition 6.11, Theorem 7.1 and Lemma 8.4]{Wu-color-equi}, we get that  
\begin{equation}\label{eq-isomorphism-Gamma-K}
H_{x^{N+1}-ax}(\Gamma_K) \cong H_{x^{N+1}-ax}(K^{(2)})\{q^2\} \oplus H_{x^{N+1}-ax}(K^{(2)}),
\end{equation}
where $K^{(2)}$ is $K$ colored by $2$. In other words, $K^{(2)}$ is $K$ drawn with wide edges. Recall that $K$ has a diagram with $c_-$ negative crossings. By the normalization of the homological grading of  $H_{x^{N+1}-ax}(K^{(2)})$ given in \cite[Definition 6.11]{Wu-color-equi}, this implies that 
\begin{equation}\label{eq-K-2-vanish}
H_{x^{N+1}-ax}^l(K^{(2)})=0 \text{ if } l \geq 2c_-+1.
\end{equation}
Combining \eqref{eq-isomorphism-Gamma-K}, \eqref{eq-K-2-vanish} and the long exact sequence \eqref{minor-long-exact-sequence}, we get that
\begin{equation}\label{eq-Gamma-K-vanish}
\mathscr{H}_{x^{N+1}-ax}^{2k}(\Gamma_K) \cong 0 \text{ if } k \geq c_-+1.
\end{equation} 
This implies that, when $k \geq c_-+1$, the homomorphism \eqref{eq-homomorphism-2cable} is surjective. Thus, homomorphism \eqref{eq-homomorphism-2cable} induces an isomorphism of the free parts of $H_{x^{N+1}-ax}^{0}(K_{2,2k-1})\{q^{2(N-1)}\}$ and $H_{x^{N+1}-ax}^{0}(K_{2,2k+1})$. By Corollary \ref{cor-H-equi-structure}, we have
\begin{equation}\label{eq-linearity-2cable-N+}
s_N(K_{2,2k+1}) = s_N(K_{2,2k-1}) + 2(N-1) \text{ if } k\geq c_-+1.
\end{equation}

Now consider the special case $N=2$. By the normalization given in \cite[Definition 6.11]{Wu-color-equi} and \cite[Lemma 7.3]{Wu-color-ras}\footnote{Although \cite[Lemma 7.3]{Wu-color-ras} is stated only for the deformed version of the colored $\mathfrak{sl}(N)$ homology, it is easy to see that this lemma and its proof are true for all three versions of the colored $\mathfrak{sl}(N)$ homology.}, one can see that unknotting $K$ does not change $H_{x^{3}-ax}(K^{(2)})$. This implies that $H_{x^{3}-ax}(K^{(2)}) \cong H_{x^{3}-ax}(U^{(2)}) \cong \C[a]$, where $U$ is the unknot and the homological grading of $\C[a]$ is $0$. So 
\begin{equation}\label{eq-K-2-vanish-2}
H_{x^{3}-ax}^l(K^{(2)})=0 \text{ if } l \geq 1
\end{equation}
and 
\begin{equation}\label{eq-Gamma-K-vanish-2}
\mathscr{H}_{x^{3}-ax}^{2k}(\Gamma_K) \cong 0 \text{ if } k \geq 1.
\end{equation} 
Repeating the argument in the previous paragraph, we get that
\begin{equation}\label{eq-linearity-2cable-2+}
s_2(K_{2,2k+1}) = s_2(K_{2,2k-1}) + 2 \text{ if } k \geq 1.
\end{equation}

Using \eqref{eq-linearity-2cable-N+}, \eqref{eq-linearity-2cable-2+} and the fact $s_N(\overline{K})=-s_N(K)$, one easily deduces that
\begin{equation}\label{eq-linearity-2cable-N-}
s_N(K_{2,2k+1}) = s_N(K_{2,2k-1}) + 2(N-1) \text{ if } k \leq -c_+-1
\end{equation}
and 
\begin{equation}\label{eq-linearity-2cable-2-}
s_2(K_{2,2k+1}) = s_2(K_{2,2k-1}) + 2 \text{ if } k \leq -1.
\end{equation}
\end{proof}

\end{document}